\documentclass[a4paper,11pt]{article}

\usepackage[utf8]{inputenc}
\usepackage[T1]{fontenc}
\usepackage[english]{babel}

\usepackage{amsthm}
\usepackage{amsmath}
\usepackage{amssymb}
\usepackage{mathrsfs}
\usepackage{soul}
\usepackage{mathabx}

\usepackage[hmargin=2.5cm,vmargin=3.5cm]{geometry}

\newcommand{\N}{\mathbb{N}}
\newcommand{\Z}{\mathbb{Z}}

\newcommand{\R}{\mathbb{R}}
\newcommand{\C}{\mathbb{C}}

\newcommand{\E}{\mathbb{E}}
\newcommand{\KG}{\mathcal{K}(G,\Gamma_{-1}(\mathcal{H}))}
\newcommand{\prodalpha}{\G\rtimes_\alpha G}
\newcommand{\G}{\Gamma_{-1}(\mathcal{H})}
\newcommand{\h}{\mathcal{H}}

\newtheorem{theorem}{Theorem}[section]
\newtheorem{definition}[theorem]{Definition}
\newtheorem{prop}[theorem]{Proposition}
\newtheorem{rmq}[theorem]{Remark}

\newtheorem{lemme}[theorem]{Lemma}

\newtheorem*{theo}{Theorem}

\begin{document}
\begin{center}
\begin{huge}
\bfseries Dilation properties of measurable Schur multipliers and Fourier multipliers
\end{huge}\\[0.7cm]
\begin{Large}
\bfseries Charles Duquet\footnote{ Laboratoire de Besançon, Université de Franche-Comté, Besançon, France, charles.duquet@univ-fcomte.fr
\textit{Key words}: dilation; Schur multipliers; Fourier multipliers; Completely positive maps

}
\end{Large}
\end{center}

\begin{Large}Abstract:\end{Large} In the article, we find new dilatation results on non-commutative $L^p$ spaces. We prove that any self-adjoint, unital, positive measurable Schur multiplier on some $B(L^2(\Sigma))$ admits, for all $1\leq p<\infty$, an invertible isometric dilation on some non-commutative $L^p$-space.
We obtain a similar result for self-adjoint, unital, completely positive Fourier multiplier on $VN(G)$, when $G$ is a unimodular locally compact group. Furthermore, we establish multivariable versions of these results.

\tableofcontents

\section{Introduction}
A famous theorem of Akcoglu \cite{Akcoglu} asserts that positive contractions on classical $L^p$-spaces, $1<p<\infty$, admit an isometric dilation, as follows:
for any measure space $(\Sigma,\mu)$, for $1<p<\infty$ and for any positive contraction $T:L^p(\Sigma)\to L^p(\Sigma)$, there exist a measure space $(\Sigma',\mu')$, two contractions $J:L^p(\Sigma)\to L^p(\Sigma')$ and $Q:L^p(\Sigma')\to L^p(\Sigma)$ and an inversible isometry $U:L^p(\Sigma')\to L^p(\Sigma')$ such that  $T^k=QU^kJ$ for all integer $k\geq 0$.

A natural question is to extend such a dilation property on non-commutative $L^p$-spaces associated with a semi-finite von Neumann algebra. In the sequel we call tracial von Neumann algebra any pair $(N,\tau)$, where $N$ is a (semi-finite) von Neumann algebra and $\tau$ is a normal semi-finite faithful trace (n.s.f in short) on $N$. In this framework, the appropriate notion of dilation is the following.

\begin{definition}
Let $(N,\tau)$ be a tracial von Neumann algebra and let $1\leq p<\infty$.
We say that an operator $T:L^p(N,\tau)\to L^p(N,\tau)$ is dilatable if there exist a tracial von Neumann algebra $(N',\tau')$, two contractions $J:L^p(N,\tau)\to L^p(N',\tau')$, $Q: L^p(N',\tau')\to L^p(N,\tau)$ and an invertible isometry $U:L^p(N',\tau')\to L^p(N',\tau')$, such that for all $k\geq 0$, $T^k=QU^kJ$.
\begin{align*}
\begin{array}{ccc}
L^p(N',\tau')&\overset{U^k}{\longrightarrow}&L^p(N',\tau')\\
J\uparrow& & \downarrow Q\\
L^p(N,\tau)&\overset{T^k}{\longrightarrow}&L^p(N,\tau)
\end{array}
\end{align*}
\end{definition}

The first result to mention is that not all positive contractions on non-commutative $L^p$-spaces are dilatable. Indeed for any $1<p\neq 2<\infty$, it has been proven in \cite{JLM} that there exists a completely positive contraction on the Schatten class $S^p$ which is not dilatable. On the opposite direction, C. Arhancet exhibited 
in \cite{A1} remarkable families of dilatable operators. The goal of this
paper is to present new classes of dilatable operators which 
extend Arhancet's results. 

We will be interested in operators which act on 
non-commutative $L^p$-spaces for all $1\leq p\leq\infty$.
Let us start with some background and definitions.
We say that a positive operator 
$T: (N,\tau)\to (N',\tau')$ between two tracial von Neumann algebras is trace preserving if for all $x\in N_+\cap L^1(N,\tau)$, we have $\tau'(T(x))=\tau(x)$. The following is well-known (see e.g. \cite[Lemma 1.1]{JX}).

\begin{lemme}\label{lem superdilatation}
Let $(N,\tau), (N',\tau')$ be two tracial von Neumann algebras
and let $T:N\to N'$ be a  positive trace preserving contraction. 
Then for all $1\leq p<\infty$, there exists a necessarily unique contraction 
$T_p: L^p(N,\tau)\to L^p(N',\tau')$ such that for all $x\in N\cap L^p(N,\tau)$, $T(x)=T_p(x)$.

If further $T$ is a one-to-one $\ast$-homomorphism, then $T_p$ is an isometry.
\end{lemme}

Assume that $J: (N,\tau)\to (N',\tau')$ is a one-to-one 
trace preserving $\ast$-homomorphism. Let $J_1 :L^1(N,\tau)\to L^1(N',\tau')$
be induced by $J$, according to Lemma \ref{lem superdilatation}. It
is well-known that $J_1^* :
N'\to N$ is a conditional expectation. In the sequel, we call it the conditional 
expectation associated with $J$.

\begin{definition} \label{rmq superdilatation}
We say that an operator
$T:(N,\tau)\to (N,\tau)$ is absolutely dilatable if there exist a
tracial von Neumann algebra $(N',\tau')$, a trace preserving one-to-one $\ast$-homomorphism $J:(N,\tau)\to (N',\tau')$ and a  trace preserving $\ast$-automorphism $U:(N',\tau')\to (N',\tau')$ such that
\begin{equation}\label{DilPpty}
T^k=\E U^kJ,\qquad k\geq 0,
\end{equation}
where $\mathbb{E}:N'\to N$ is the  the conditional 
expectation associated with $J$.
\end{definition}

A absolutely dilation is trace preserving. Indeed with the above notation, (\ref{DilPpty}) with $k=1$ yields $T=\mathbb{E}UJ$.

If $T$ satisfies Definition \ref{rmq superdilatation}, then applying Lemma 
\ref{lem superdilatation}
we obtain that for all $1\leq p<\infty$, 
$J$ (resp. $\mathbb{E}$) induces  a contraction
$L^p(N,\tau)\to L^p(N',\tau')$ (resp. $L^p(N',\tau')\to L^p(N,\tau)$)
and that $U$ induces an invertible isometry $L^p(N',\tau')\to L^p(N',\tau')$.
Moreover (\ref{DilPpty}) holds true on $L^p$-spaces.
We therefore obtain the following lemma.

\begin{lemme}\label{lem absdila implique dila p}
If $T:(N,\tau)\to (N,\tau)$ is absolutely dilatable, then for every $1\leq p <\infty$,
$T$ induces a contraction $T_p:L^p(N,\tau)\to L^p(N,\tau)$
and $T_p$ is dilatable.
\end{lemme}

In the article \cite{A1}, C. Arhancet proved that any 
self-adjoint unital positive Schur multiplier $B(l^2_{\Z})\to B(l^2_{\Z})$ is absolutely dilatable
and that whenever $G$ is discrete group,
any self-adjoint unital completely positive Fourier multiplier 
$VN(G)\to VN(G)$ is absolutely dilatable. 
Our main results are generalisations of these results. On the one hand (Theorem 
\textbf{A}), we will consider measurable Schur multipliers. On the
other hand (Theorem \textbf{B}), we will consider 
Fourier multipliers on unimodular groups.

Let $\Sigma$ be a $\sigma$-finite measure space with respect to a
measure simply denoted by $dt$. Let $S^2(L^2(\Sigma))$ be the space
of all Hilbert-Schmidt operators on $L^2(\Sigma)$. 
We recall that for any
$f\in L^2(\Sigma^2)$, we may define a bounded operator
\begin{equation}\label{Sf}
S_f : L^2(\Sigma^2)\longrightarrow L^2(\Sigma^2),
\qquad S_f(h) =\int_\Sigma f(\cdot,t)h(t)dt,
\end{equation}
that $S_f$ is a Hilbert-Schmidt operator, and that the mapping $f\mapsto S_f$
is a unitary from $L^2(\Sigma^2)$ onto $S^2(L^2(\Sigma))$.

For any $\varphi\in L^\infty(\Sigma^2)$, we denote by $M_\varphi: S^2(L^2(\Sigma))\to S^2(L^2(\Sigma))$ the operator defined by
\begin{equation}\label{Schur}
M_\varphi(S_f) = S_{\varphi f},\qquad f\in L^2(\Sigma^2).
\end{equation}
We say that $M_\varphi$ (or simply $\varphi$) 
is a (measurable) Schur multiplier 
on $B(L^2(\Sigma))$, if there exists $C> 0$ such that
\begin{equation}\label{Bounded}
\|M_\varphi(S)\|_{B(L^2(\Sigma))}
\leq C\|S\|_{B(L^2(\Sigma))},
\qquad S\in S^2(L^2(\Sigma)).
\end{equation}

Our first main result is the following.

\begin{theo}[\textbf{A}]
Let $\Sigma$ be a $\sigma$-finite measure space. Let $\varphi\in L^\infty(\Sigma^2)$ and assume that $M_\varphi$ is a self-adjoint, unital, positive Schur multiplier on $B(L^2(\Sigma))$.
Then $M_\varphi$ is absolutely dilatable.
\end{theo}

Let $G$ be a unimodular locally compact group. Let $\lambda\colon G\to B(L^2(G))$
be the left regular representation of $G$, that is, 
$\bigl[\lambda(s)f\bigr](t)=f(s^{-1}t)$
for any $f\in L^2(G)$ and $s,t\in G$.
We let $VN(G)=\lambda(G)''\subset B(L^2(G))$ be the 
group von Neumann algebra of $G$.
We denote by 
$$
\omega_G: VN(G)_+\longrightarrow [0;+\infty]
$$
the Plancherel weight on $VN(G)$, which is a trace in the unimodular case.
We say that a $w^*$-continuous operator $T:VN(G)\to VN(G)$ is a Fourier multiplier, if there exists a bounded continuous function 
$u:G\to \C$ such that for all $s\in G$, $T(\lambda(s))=u(s)\lambda(s)$.
In this case, $u$ is necessarily unique, we write $T=T_u$ and
$u$ is called the symbol of $T_u$.
In the unimodular case, 
we say that an operator $T: VN(G)\to VN(G)$ is self-adjoint if 
$$
\omega_G(T(x)y^*)=\omega_G(xT(y)^*),
\qquad x,y\in VN(G)\cap L^1(VN(G),\omega_G).
$$

Our second main result is the following.

\begin{theo}[\textbf{B}]
Let $T_u: VN(G)\to VN(G)$ be a self-adjoint, unital, completely positive Fourier multiplier.
Then $T_u$ is absolutely dilatable.
\end{theo}

We mention that in the two recent papers
\cite{A2} and \cite{A3}, C. Arhancet establishes dilation properties 
for $w^*$-continuous semi-groups of either 
measurable Schur 
multipliers, or Fourier multipliers on groups. 
However the techniques in these papers
do not apply to single operators. Consequently
the proofs of Theorems \textbf{A} and
\textbf{B} require different tools.

The last part of this paper is devoted to multivariable 
versions of Theorems \textbf{A} and
\textbf{B}. 
We obtain a dilation property for any $n$-tuple of either
Schur multipliers satisfying the assumptions of 
Theorems \textbf{A}, or 
Fourier multipliers  satisfying the assumptions of 
Theorems \textbf{B}. We will establish the following two
results.

\begin{theo}[\textbf{C}]
Let $\varphi_1,\dots,\varphi_n\in L^\infty(\Sigma^2)$ and assume each $M_{\varphi_i}$ is a self-adjoint, unital, positive Schur multiplier on $B(L^2(\Sigma))$. Then there exist a tracial von Neumann algebra $(\mathcal{M},\tau)$, a commuting $n$-tuple $(U_1,\dots,U_n)$ of trace preserving $\ast$-automorphisms on $\mathcal{M}$ and a trace preserving one-to-one $\ast$-homomorphism $J: B(L^2(\Sigma))\to\mathcal{M}$ such that 
\begin{align*}
M_{\varphi_1}^{k_1}\cdots M_{\varphi_n}^{k_n}=
\mathbb{E}U_1^{k_1}\cdots U_n^{k_n}J
\end{align*}
for all $k_i\in \N_0$, $1\leq i\leq n$, where 
$\mathbb{E}:M\to B(L^2(\Sigma))$ is  the conditional 
expectation associated with $J$.
\end{theo}

\begin{theo}[\textbf{D}]
Let $G$ be a unimodular locally compact group and 
let $T_{u_1},\ldots, T_{u_n}$ be self-adjoint, unital, completely positive Fourier multipliers on $VN(G)$. Then 
there exist a tracial von Neumann algebra $(\mathcal{M},\tau)$, a commuting $n$-tuple $(U_1,\dots,U_n)$ of trace preserving $\ast$-automorphisms on $\mathcal{M}$ and a trace preserving one-to-one $\ast$-homomorphism 
$J: VN(G)\to \mathcal{M}$ such that 
\begin{align*}
T_{u_1}^{k_1}\cdots T_{u_n}^{k_n}=\mathbb{E}U_1^{k_1}\cdots U_n^{k_n}J
\end{align*}
for all $k_i\in \N_0$, $1\leq i\leq n$, where 
$\mathbb{E}:M\to B(L^2(\Sigma))$ is the conditional 
expectation associated with $J$.
\end{theo}

\section{Preliminaries, $L^p_\sigma$-spaces and  Fermions}
For any Banach space $X$ and any $(x,x^*)\in X\times X^*$, we let 
$\langle x^*,x\rangle$ denote the action of $x^*$ on $x$. Whenever $H$ is a Hilbert space, we let $\langle\cdot,\cdot\rangle_H$
denote the inner product on $H$, that we assume linear
in the first variable and anti-linear in the second variable.

Let $\Sigma$ be a $\sigma$-finite measure space and let $X$ be a Banach space. For any $1\leq p\leq +\infty$, we let 
$L^p(\Sigma,X)$ denote the classical Bochner space of strongly
measurable function $f:\Sigma\to X$ (defined up to almost everywhere zero functions) such that the norm function $t\mapsto \|f(t)\|$ 
belongs to $L^p(\Sigma)$, equipped
with $\Vert f\Vert_p = \Vert t\mapsto \|f(t)\|\Vert_{L^p(\Sigma)}$.
(see \cite[p. 49-50]{Diesel} for more explanations). We will use the
fact that if $p$ is finite, then $L^p(\Sigma)\otimes X$ is
a dense subspace of $L^p(\Sigma,X)$.

Throughout we let $L^p_\R(\Sigma)$ denote the
subset of real-valued functions in $L^p(\Sigma)$.

We denote by $X\hat{\otimes}Y$ the projective 
tensor product of any two Banach spaces $X$ and $Y$ (see \cite{Ryan}).
We recall the isometric isomorphisms
\begin{equation}\label{L1}
L^1(\Sigma,Y)\simeq L^1(\Sigma)\hat{\otimes}Y
\qquad\hbox{and}\qquad
(L^1(\Sigma)\hat{\otimes}Y)^*\simeq B(L^1(\Sigma),Y^*),
\end{equation}
see \cite[example 2.19]{Ryan} and discussion p. 24 of the same book.

For any 
$x,y\in L^2(\Sigma)$, we consider $x\otimes y$ as an
element of $L^2(\Sigma^2)$ by writing 
$(x\otimes y)(s,t)=x(s)y(t)$ for $(s,t)\in\Sigma^2$.
Then using (\ref{Sf}), we have
$$
S_{x\otimes y}(h) =\left(\int\limits_\Sigma y(t)h(t)dt \right)x,
\qquad h\in L^2(\Sigma).
$$
Very often we will identify $x\otimes y$
and $S_{x\otimes y}$. Thus 
$L^2(\Sigma)\otimes L^2(\Sigma)$ is regarded as the
space of finite rank operators on $L^2(\Sigma)$. This extends to an 
isometric isomorphism
\begin{equation}\label{S1}
L^2(\Sigma)\hat{\otimes}L^2(\Sigma)\simeq S^1(L^2(\Sigma)),
\end{equation}
where $S^1(L^2(\Sigma))$ denotes the Banach space of trace
class operators on $L^2(\Sigma)$.

Another useful result on the projective 
tensor product is the isometry $L^1(\Sigma_1)\hat{\otimes}L^1(\Sigma_2)
\simeq L^1(\Sigma_1\times \Sigma_2)$.

Let $N$ be a von Neumann algebra. 
We recall that the product is separately $w^*$-continuous 
on $N$. More precisely, 
for any $y\in N$ and $\eta\in N_*$, 
the map $z\mapsto \langle yz,\eta \rangle$ from $N$ into $\C$ is continuous. 
We let $\eta y\in N_*$ such that 
$\langle yz,\eta\rangle= \langle z,\eta y \rangle$
for all $z\in N$. Likewise, the map $z\mapsto \langle zy,\eta\rangle$ from $N$ to $\C$ is
continuous and we 
let $y\eta\in N_*$ such that 
$\left\langle zy,\eta\right\rangle=\left\langle z,y\eta \right\rangle$
for all $z\in N$.

We now turn to a dual variant of the Bochner spaces $L^p(\Sigma,N)$. We mostly follow
\cite{Hensen}, to which we refer for more references and details.
We say that a function $f:\Sigma\to N$ is $w^*$-measurable if,
for all $\eta\in N_*$, the function $t\mapsto \langle f(t),\eta\rangle$ is measurable.
Fix some $1< q\leq\infty$.
We recall that every order bounded subset of $L^q_\R(\Sigma)$ has a supremum in $L^q_\R(\Sigma)$, denoted by $L^q-\sup$. 
We define 
\begin{align*}
\mathcal{L}^q_\sigma(\Sigma,N):=\left\lbrace \begin{array}{cc}&f:\Sigma\to N
\text{ } w^*\text{-measurable }; \langle f,\eta\rangle\in L^q(\Sigma)\text{ } \forall \eta\in N_*\\
&\text{ and } \lbrace |\langle f,\eta\rangle|:\|\eta\|\leq 1\rbrace\text{ is order bounded in }L^q_\R(\Sigma)\end{array}\right\rbrace.
\end{align*}
We define, for all $f\in \mathcal{L}^q_\sigma(\Sigma,N)$, the semi-norm 
$$
\Vert f\Vert_{\mathcal{L}^q} = 
\bigl\Vert L^q-\sup\lbrace |\langle f,\eta\rangle:\|\eta\|\leq 1\rbrace\bigr\Vert_{L^q(\Sigma)}.
$$
The kernel of this semi-norm is
\begin{align*}
N_\sigma:=\lbrace f\in \mathcal{L}^q_\sigma(\Sigma,N): 
\text{ }\forall \eta\in N_*, \langle f,\eta\rangle=0\text{ a.e.}\rbrace.
\end{align*}
We set $L^q_\sigma(\Sigma,N)=\mathcal{L}^q_\sigma(\Sigma,N)/N_\sigma$. This is
a Banach space for the resulting norm, which we simply denote by
$\|.\|_q$. We note that $L^q(\Sigma,N)\subset L^q_\sigma(\Sigma,N)$ isometrically and
that $L^q(\Sigma)\otimes N$ is $w^*$-dense in $L^q_\sigma(\Sigma,N)$.

Let $1\leq p\leq \infty$
such that $\dfrac{1}{p}+\dfrac{1}{q}=1$. For any $h\in L^p(\Sigma,N_*)$ 
and $f\in L^q_\sigma(\Sigma,N)$, the function 
$t\mapsto\langle f(t),h(t)\rangle$ belongs 
to $L^1(\R)$ and satisfies
$\int_\Sigma|\langle f(t),h(t)\rangle| dt\leq \|h\|_p\|f\|_q$. 
This allows to define
\begin{equation}\label{Dual}
\langle f,h\rangle = \int_\Sigma \langle f(t),h(t)\rangle dt,
\qquad h\in
L^p(\Sigma,N_*),\ f\in L^q_\sigma(\Sigma,N).
\end{equation}

\begin{theorem}\label{th Lpsigma predual Lq }
Let $1\leq p<\infty$ and $1<q\leq \infty$ such that 
$\dfrac{1}{p}+\dfrac{1}{q}=1$. Then the duality pairing (\ref{Dual})
extends to an isometric isomorphism
\begin{align*}
L^p(\Sigma,N_*)^*\simeq L^q_\sigma(\Sigma,N).
\end{align*}
\end{theorem}
This theorem is due to Bukhvalov, see  \cite[Theorem 4.1]{Buka2}, \cite[Theorem 0.1]{Buka3}.

We have isometric identifications
\begin{equation}\label{th identification lsigmainfty ...}
L^\infty(\Sigma)\overline{\otimes}N\overset{(ii)}{\simeq}
L^\infty_\sigma(\Sigma,N)\overset{(i)}{\simeq} B(L^1(\Sigma),N),
\end{equation}
where $L^\infty(\Sigma)\overline{\otimes}N$ denotes
the von Neumann
tensor product of $L^\infty(\Sigma)$ and $N$.
The identification (i) follows from (\ref{L1}) and 
Theorem \ref{th Lpsigma predual Lq }. More explicitly, 
for any $F\in L^\infty_\sigma(\Sigma,N)$, the associated 
operator $T\in B(L^1(\Sigma),N)$ provided by (i) is given by
\begin{align}\label{equalite eta}
\forall h\in L^1(\Sigma),\text{ }\forall \eta\in N_*,\text{ } \langle T(h),\eta\rangle=\int_\Sigma\langle F(t),\eta\rangle h(t)dt.
\end{align}
The identification (ii) is proved in \cite{Sakai}. 
It is a combination of \cite[Definition 1.22.10]{Sakai} and \cite[Theorem 1.22.12]{Sakai}.

Concerning  $L^\infty(\Sigma)\overline{\otimes}N$, we 
mention that if $N$ is equipped with a n.s.f. trace 
$\tau$, then we equip $L^\infty(\Sigma)\overline{\otimes}N$
with the unique n.s.f. trace
$\tilde{\tau}=\int\,\cdotp\overline{\otimes}\tau$
such that $\tilde{\tau}(h\otimes x)=\tau(x)\int_\Sigma h(t)dt$
for all $h\in L^\infty(\Sigma)_+$ and $x\in N_+$. Then 
for all $1\leq p<\infty$,
\begin{equation}\label{Lp}
L^p(L^\infty(\Sigma)\overline{\otimes}N,\tilde{\tau})\simeq 
L^p(\Sigma,L^p(N,\tau)).
\end{equation}
Indeed let 
\begin{equation}\label{E}
{\mathcal E}=\text{span}\bigl\{\chi_E : E
\ \text{measurable},\ \vert E\vert<\infty\bigr\}.
\end{equation}
Then ${\mathcal E}
\otimes (N\cap L^1(N))$ is both dense in 
$L^p(L^\infty\overline{\otimes}N,\tilde{\tau})$ and in $L^p(\Sigma,L^p(N))$, 
and the norms of $L^p(L^\infty\overline{\otimes}N,\tilde{\tau})$ and 
$L^p(\Sigma,L^p(N))$
coincide on this subspace.

\begin{lemme}\label{lem RcircF}
Let $V:N_1\to N_2$ be a $w^*$-continuous operator, 
where $N_1$ and $N_2$ are two von Neumann algebras. Then the map $V_\circ:L_\sigma^\infty(\Sigma,N_1)\to L^\infty_\sigma(\Sigma,N_2)$ given for all $F\in L_\sigma^\infty(\Sigma,N_1)$ by $V_\circ(F)=V\circ F$ is  well-defined 
and  $w^*$-continuous. Moreover,  $\|V_\circ\|=\|V\|$.
\end{lemme}

\begin{proof}
The operator $V$ has a pre-adjoint 
$V_*:N_{2*}\to N_{1*}$. It is plain that $Id_{L_1(\Sigma)} \otimes V_*$
extends to a bounded map 
$Id_{L_1(\Sigma)}\overline{\otimes}V_* :
L^1(\Sigma,N_{2*})\to L^1(\Sigma,N_{1*})$, with norm
$\|Id_{L^1(\Sigma)}\overline{\otimes}V_*\|=\|V_*\|$. Then using
Theorem \ref{th Lpsigma predual Lq }, $V_\circ$ coincides with $\left(Id_{L_1(\Sigma)}\overline{\otimes}V_*\right)^*$.
\end{proof}

The product of two $N$-valued measurable functions is measurable. 
However the product of two $N$-valued  $w^*$-measurable functions is not 
necessarily $w^*$-measurable. We need the following to circumvent this difficulty.

\begin{lemme}\label{prop produit résonable}
Let  $T_1\in B(L^1(\Sigma);N)$ and $T_2\in B(L^1(\Sigma);N)$.
Then there exists a unique $P\in B(L^1(\Sigma^2),N)$ such that 
\begin{align*}
P(h_1\otimes h_2)=T_1(h_1)T_2(h_2),
\qquad h_1,h_2\in L^1(\Sigma).
\end{align*}
\end{lemme}

\begin{proof}
We define  $Q:(h_1,h_2)\mapsto T_1(h_1)T_2(h_2)$ from $L^1(\Sigma)\times L^1(\Sigma)$ into $N$. This is a bounded bilinear map, with $\|Q\|\leq \|T_1\|\|T_2\|$. 
Hence there exists a bounded linear map $P:L^1(\Sigma)\hat{\otimes} L^1(\Sigma)\to N$ such that for all $h_1,h_2\in L^1(\Sigma)$,
$P(h_1\otimes h_2)=T_1(h_1)T_2(h_2)$.
Using $L^1(\Sigma)\hat{\otimes}L^1(\Sigma)\simeq
L^1(\Sigma^2)$, this
yields the result.
\end{proof}

Let $\varphi_1,\varphi_2\in L^\infty_\sigma(\Sigma,N)$.
Let $T_1, T_2\in B(L^1(\Sigma);N)$ be the representatives of $\varphi_1$, 
and $\varphi_2$, respectively, through the identification (\ref{th identification lsigmainfty ...}), (i).
We let 
\begin{align}
\tilde{\varphi_1\times\varphi_2}\in L^\infty_\sigma(\Sigma^2,N )
\end{align}
be the representative of 
the operator $P\in B(L^1(\Sigma^2),N)$ given by Lemma 
\ref{prop produit résonable}. Thus for all $h_1,h_2\in L^1(\Sigma)$ and $\eta\in N_*$,
\begin{align}\label{eq phitildetimes phi}
\langle T_1(h_1)T_2(h_2),\eta\rangle=\int_{\Sigma^2}\langle \tilde{\varphi_1\times\varphi_2}(s,t),\eta\rangle h_1(s)h_2(t)dtds.
\end{align}

\begin{rmq}\label{rmq produit tilde simple}
When $\varphi_1,\varphi_2\in L^\infty(\Sigma,N)$, the product 
$\tilde{\varphi_1\times\varphi_2}$ is simply given by
$\varphi_1\tilde{\times}\varphi_2(s,t)=
\varphi_1(s)\varphi_2(t)$,
since $\varphi_1,\varphi_2$ are both measurable.
\end{rmq}

\begin{lemme}\label{prop existence gamma}
There exists a unique $w^*$-continuous contraction $\Gamma:L^2_\sigma(\Sigma^2,N)\to B(L^2(\Sigma))\overline{\otimes}N$ such that for all $\theta\in L^2(\Sigma^2)$ and for all $y\in N$, $\Gamma(\theta\otimes y)=S_\theta\otimes y$.
\end{lemme}

\begin{proof}
Recall that $L^2_\sigma(\Sigma^2,N)=L^2(\Sigma^2,N_*)^*$, by Theorem
\ref{th Lpsigma predual Lq }. Let $|\|.\||$ denote the norm on
$S^1(L^2(\Sigma))\otimes N_{*}$ induced by $(B(L^2(\Sigma))\overline{\otimes}N)_*$.
Recall ${\mathcal E}\subset L^2(\Sigma)$ defined by (\ref{E}). 
Then under the identification
(\ref{S1}), $L^2(\Sigma)\otimes L^2(\Sigma)\otimes N_*$ is dense in
$(B(L^2(\Sigma))\overline{\otimes}N)_*$, hence 
${\mathcal E}\otimes {\mathcal E} \otimes N_*$ is dense in
$(B(L^2(\Sigma))\overline{\otimes}N)_*$.

Let $L:{\mathcal E}\otimes {\mathcal E}\otimes N_*\to L^2(\Sigma^2,N_*)$ 
be the linear mapping such that for all $u, v\in {\mathcal E}$ and $\eta\in N_*$, 
\[ 
\bigl[L(u\otimes v\otimes \eta)\bigr] (s,t) =\,u(s)v(t)\eta,
\qquad (s,t)\in\Sigma^2.
\]
Let $\Phi\in{\mathcal E}\otimes {\mathcal E}\otimes N_*$. It can be written
as $\Phi=\sum\limits_{i,j=1}^{M}
\chi_{E_i}\otimes\chi_{F_j}\otimes\eta_{ij},$ 
where $M\geq 1$ is an integer, $E_1,\ldots E_{M}$ are pairwise disjoint elements
such that $0<\vert E_i\vert <\infty$ for all $i$,
$F_1,\ldots F_{M}$ are pairwise disjoint elements
such that $0<\vert F_j\vert <\infty$ for all $j$, and $\eta_{ij}\in N_{*}$.
We define, for all $i,j=1,\ldots,M$, 
$z_{ij}=|E_i|^\frac{1}{2}|F_j|^\frac{1}{2}\eta_{ij}\in N_*$. 
By classical duality, there exists a family 
$(m_{ij})_{1\leq i,j\leq M}\subset N$ such that 
$$
\left(\sum_{i,j=1}^M\|z_{ij}\|^2\right)^{\frac{1}{2}}=
\sum_{i,j=1}^M\left\langle z_{ij},m_{ij}\right\rangle 
\qquad\text{ and }\qquad 
\sum_{i,j=1}^M\|m_{ij}\|^2\leq 1.
$$
Thanks to the disjointness of the $E_i$ and of the $F_j$, we have
\begin{align} \label{sum}
\|L(\Phi)\|^2_{L^2(\Sigma^2,N_*)}=
\sum_{ij}\|\eta_{ij}\|^2|E_i||F_j|=\sum_{i,j}\|z_{ij}\|^2  .
\end{align}
We now define 
$$
\Psi=\sum\limits_{i,j=1}^M\dfrac{\chi_{E_i}}{|E_i|^\frac{1}{2}}\otimes \dfrac{\chi_{F_j}}{|F_j|^\frac{1}{2}}\otimes m_{ij}\in L^2(\Sigma)\otimes 
L^2(\Sigma)\otimes N\subset B(L^2(\Sigma))\overline{\otimes}N.
$$
We have
$$
\left\langle \Phi,\Psi\right\rangle = \sum_{i,j} 
\left\langle z_{ij},m_{ij}\right\rangle = \sum_{i,j}\|z_{ij}\|^2
=\|L(\Phi)\|_{L^2(\Sigma^2,N_*)},
$$
by $\eqref{sum}$. Hence
$$
\|L(\Phi)\|_{L^2(\Sigma^2,N_*)}\leq \|\Psi\|_{B(L^2(\Sigma))\overline{\otimes}N}|\|\Phi\||.
$$
We will prove below that 
\begin{equation}\label{Goal}
\|\Psi\|_{B(L^2(\Sigma))\overline{\otimes}N}\leq 1.
\end{equation}
Taking this for granted, we obtain that $L$ extends to a contraction (still denoted by)
$$
L\colon \left(B(L^2(\Sigma))\overline{\otimes}N\right)_*
\longrightarrow L^2(\Sigma^2,N_*).
$$
Its adjoint $\Gamma=L^*$ is a $w^*$-continuous contraction
from $L^2_\sigma(\Sigma^2,N)$ into $B(L^2(\Sigma))\overline{\otimes}N$
and it is easy to check that $\Gamma(\theta\otimes y)=S_\theta\otimes y$
for all $\theta\in L^2(\Sigma^2)$ and $y\in N$. This proves the existence result.
The uniqueness comes from the $w^*$-continuity of $\Gamma$ 
and the $w^*$-density 
of $L^2(\Sigma^2)\otimes N$.

It therefore remains to check (\ref{Goal}). Let $K$ be a Hilbert space
such that $N\subset B(K)$ as a von Neumann algebra. Then
$$
B(L^2(\Sigma))\overline{\otimes}N\subset 
B(L^2(\Sigma))\overline{\otimes}B(K)\,\simeq\, B(L^2(\Sigma,K)).
$$
Let $\xi,\zeta\in L^2(\Sigma,K)$. Then 
$$
\langle\Psi(\xi),\zeta\rangle\,=\,
\sum_{i,j=1}^M \,\frac{1}{\vert E_i\vert^\frac12\vert F_j\vert^\frac12}\,
\int_{E_i\times F_j} \langle m_{ij}(\xi(t)),\zeta(s)\rangle dtds.
$$
Hence by Cauchy-Schwarz,
\begin{align*}
\vert\langle\Psi(\xi),\zeta\rangle\vert \,
&\leq\,
\sum_{i,j=1}^M \,\frac{1}{\vert E_i\vert^\frac12\vert F_j\vert^\frac12}\,
\Vert m_{ij}\Vert \Bigl(\int_{E_i}\Vert \xi(t)\Vert dt\Bigr)
\Bigl(\int_{F_j}\Vert \zeta(s)\Vert ds\Bigr)\\
&\leq\,
\biggl(\sum_{i,j=1}^M \Vert m_{ij}\Vert^2\biggr)^\frac12
\biggl(\sum_{i,j=1}^M\frac{1}{\vert E_i\vert\vert F_j\vert}
 \Bigl(\int_{E_i}\Vert \xi(t)\Vert dt\Bigr)^2
\Bigl(\int_{F_j}\Vert \zeta(s))\Vert ds\Bigr)^2\biggr)^\frac12\\
&=\,
\biggl(\sum_{i,j=1}^M \Vert m_{ij}\Vert^2\biggr)^\frac12
\biggl(\sum_{i}^M\frac{1}{\vert E_i\vert}
\Bigl(\int_{E_i}\Vert \xi(t)\Vert dt\Bigr)^2
\biggr)^\frac12\biggl(\sum_{j=1}^M\frac{1}{\vert F_j\vert}
\Bigl(\int_{F_j}\Vert \zeta(s))\Vert ds\Bigr)^2\biggr)^\frac12.
\end{align*}
By Cauchy-Schwarz again we have
$$
\frac{1}{\vert E_i\vert}
\Bigl(\int_{E_i}\Vert \xi(t)\Vert dt\Bigr)^2\,\leq\,\int_{E_i}
\Vert \xi(t)\Vert^2 dt
\qquad\hbox{and}\qquad
\frac{1}{\vert F_j\vert}
\Bigl(\int_{F_j}\Vert \zeta(s)\Vert dt\Bigr)^2\,\leq\,\int_{F_j}
\Vert \zeta(s)\Vert^2 ds
$$
for all $i,j=1,\ldots, M$. We derive
$$
\vert\langle\Psi(\xi),\zeta\rangle\vert \,\leq
\biggl(\int_{\Sigma}\Vert\xi(t)\Vert^2dt\biggl)^\frac12
\biggl(\int_{\Sigma}\Vert\zeta(s)\Vert^2ds\biggl)^\frac12,$$
which proves (\ref{Goal}). 
\end{proof}

In the rest of this section we give some background on 
antisymmetric Fock spaces and Fermions. 
We refer to \cite{Reffermions1} and \cite{Reffermions2} for more details
and information.
Let $H$ be a real Hilbert space and we let $H_\C$
denote its complexification. For any integer $n\geq 1$, we let 
$\Lambda_n(H_\C)$ denote the $n$-fold antisymmetric space
over $H_\C$, equipped with the inner product defined by 
\begin{align*}
\langle h_1\wedge \cdots \wedge h_n,k_1\wedge \cdots \wedge k_n\rangle_{-1}=\det \left[\langle h_i,k_j\rangle_{H_\C}\right],
\qquad h_i,k_j\in H_\C.
\end{align*}
We also set $\Lambda_0(H_\C)=\C$. 
The antisymmetric Fock space over $H_\C$ is the Hilbertian
direct sum
\begin{align*}
\mathcal{F}_{-1}(H)\,=\,\bigoplus_{n\geq 0} \Lambda_n(H_\C).
\end{align*}
We let $\Omega$
be a fixed unit element of $\Lambda_0(H_\C)$.

For all $e\in H$, we recall the creation operator
$l(e):\mathcal{F}_{-1}(H)\to\mathcal{F}_{-1}(H)$
satisfying $l(e)\Omega=e$ and 
$l(e)(h_1\wedge \cdots \wedge h_n) =e\wedge
h_1\wedge \cdots \wedge h_n$
for all $h_1,\ldots,h_n\in H_\C$. 
These operators satisfy the following relation:
\begin{align}\label{CAR}
l(f)^*l(e)+l(e)l(f)^*=\langle f,e\rangle_H Id_{\mathcal{F}_{-1}(H)},
\qquad e,f\in H.
\end{align}
Let $\omega(e):\mathcal{F}_{-1}(H)\to \mathcal{F}_{-1}(H)$ 
be the self-adjoint operator
\begin{align*}
\omega(e)=l(e)+l(e)^*.
\end{align*}
It follows from (\ref{CAR}) that 
\begin{align}\label{eq équilité norme}
\omega(e)^2 = \Vert e\Vert^2_H Id_{\mathcal{F}_{-1}(H)},
\qquad e\in H.
\end{align}
By definition, 
$$
\Gamma_{-1}(H)\subset B(\mathcal{F}_{-1}(H))
$$
is the von Neumann algebra generated by 
$\lbrace \omega(e);$ $e\in H\rbrace$.
This is a finite von Neumann algebra equipped with a trace
$\tau$ defined by $\tau(x)=\langle x\Omega, \Omega\rangle_{\mathcal{F}_{-1}(H)}$ for
all $x\in \Gamma_{-1}(H)$. Clearly we have
\begin{equation}\label{Trace}
\tau(\omega(e)\omega(f))=\langle e,f\rangle_H,
\qquad e,f\in H.
\end{equation}
The space $\Gamma_{-1}(H)$ is called the 
Fermion algebra over $H$.

Let $H$ and $K$ be two real Hilbert spaces and 
let $T:H\to K$  be a contraction with complexification 
$T_\C:H_\C\to K_\C$. There exists a necessarily
unique linear contraction
$$
F_{-1}(T):\mathcal{F}_{-1}(H)\longrightarrow\mathcal{F}_{-1}(K)
$$
such that $F_{-1}(T)\Omega=\Omega$ and
$F_{-1}(T)(h_1\wedge\cdots\wedge h_n)=T_\C(h_1)\wedge \cdots \wedge T_\C(h_n)$ for all
$h_1,\ldots, h_n\in H_\C$. Next, 
there exists a necessarily unique normal, unital, 
completely positive and trace preserving map 
$$
\Gamma_{-1}(T):\Gamma_{-1}(H)\longrightarrow \Gamma_{-1}(H)
$$
such that for every $x\in \Gamma_{-1}(H)$, we have:
\begin{align*}
(\Gamma_{-1}(T)(x))\Omega=F_{-1}(T)(x\Omega).
\end{align*} 
If further $T:H\to K$ is a isometry (resp. an onto 
isometry), then $\Gamma_{-1}(T)$ is a one-to-one $\ast$-homomorphism
(resp. a one-to-one $\ast$-isomomorphism).

In particular we have
\begin{equation}\label{omega}
[\Gamma_{-1}(T)](\omega(e))= \omega(T(e)),
\qquad e\in H.
\end{equation}

In the next lemma, we consider an integer $k\geq 1$
and we let $\mathcal{P}_2(2k)$ be
the set of 2-partitions of the set $\lbrace 1,2,
\dots,2k\rbrace$.
Then for any  $\nu\in \mathcal{P}_2(2k)$, we 
let $c(\nu)$ denote the number of crossings of $\nu$.
We refer to \cite{Effros} for details.
According to Corollary 2.1 in the latter paper,
we have the following lemma, in which 
(2) is a straightforward consequence of (1).

\begin{lemme}\label{prop trace w w avec produit scalaire}
Let $(f_i)_{i=1}^{2k}$ be a family of $H$.
\begin{enumerate}
\item [(1)] We have
\begin{align*}
\tau(\omega(f_1)\cdots\omega(f_{2k}))=
\sum_{\nu\in\mathcal{P}_2(2k)}(-1)^{c(\nu)}\prod_{(i,j)\in \nu}\langle f_i,f_j\rangle_H.
\end{align*}
\item [(2)] If  for all $1\leq i<j\leq 2k$ such that $j\neq 2k-i+1$, we have $\langle f_i,f_j\rangle_H=0$, then
\begin{align*}
\tau(\omega(f_1)\cdots\omega(f_{2k}))=\langle f_1,f_{2k}\rangle_H\cdots\langle f_k,f_{k+1}\rangle_H
\end{align*}
\end{enumerate}
\end{lemme}

\section{Properties of measurable Schur multipliers}
This section is devoted to preliminaries on measurable Schur multipliers
and characterizations of some of their possible properties (positivity, unitality,
etc.). Throughout we let $\Sigma$ be a $\sigma$-finite
measure space. Given any $\varphi\in L^\infty(\Sigma^2)$, recall
$M_\varphi: S^2(L^2(\Sigma))\to S^2(L^2(\Sigma))$ defined by
(\ref{Schur}). It is plain that 
\begin{align}\label{lem norm phi}
\|M_\varphi\|_{S^2(L^2(\Sigma))}=\|\varphi\|_{L^\infty(\Sigma^2)}.
\end{align}
We recall that $\varphi$ (or $M_\varphi$) is called a 
Schur multiplier on $B(L^2(\Sigma))$ if it satisfies
(\ref{Bounded}) for some $C>0$.

In the sequel, we set $\widecheck{\varphi}(s,t)=\varphi(t,s)$
for $(s,t)\in\Sigma^2$.

\begin{lemme} \label{th equivalent schur multipliers}
Let $\varphi\in L^\infty(\Sigma^2)$, 
the following assertions are equivalent:
\begin{enumerate}
\item $\varphi$ is a Schur multiplier on $B(L^2(\Sigma))$;
\item $M_\varphi$ extends to a $w^*$-continuous operator $M_\varphi^\infty:B(L^2(\Sigma))\to B(L^2(\Sigma))$;
\item $M_\varphi$ restricts to a bounded operator $M_\varphi^1:S^1(L^2(\Sigma))\to S^1(L^2(\Sigma))$;
\item $M_\varphi$ extends to a  bounded operator $M_\varphi^0:K(L^2(\Sigma))\to K(L^2(\Sigma))$, where $K(L^2(\Sigma))$ denotes the Banach space
of compact operators on $L^2(\Sigma)$;
\item $\widecheck{\varphi}$ verifies (3).
\end{enumerate}
In this case, $M_{\widecheck{\varphi}}^1=\left(M_\varphi^0\right)^*$, $\left(M_{\widecheck{\varphi}}^1\right)^*=M_\varphi^\infty$ and $\|M_\varphi^1\|=\|M_{\widecheck{\varphi}}^1\|=\|M_\varphi^\infty\|=\|M_\varphi^0\|$.\\
\end{lemme}

\begin{proof}
The flip mapping $\sum x_k\otimes y_k\mapsto \sum y_k\otimes x_k$ extends to 
an isometric automorphism $\rho :L^2(\Sigma)\hat{\otimes}L^2(\Sigma)
\to L^2(\Sigma)\hat{\otimes}L^2(\Sigma)$. 
Using this notation and (\ref{S1}), it is easy to check 
that 
$$
M_{\widecheck{\varphi}}(S_f) = \rho\bigl(M_\varphi(S_{\rho(f)})\bigr),
\qquad f\in L^2(\Sigma)\otimes L^2(\Sigma),
$$
provided that property 3 holds true. We easily deduce that 
$5\Leftrightarrow 3$.

Let $\text{tr}$ denote the usual trace on $B(L^2(\Sigma))$.
We remark that $S_f^*=S_{\overline{\dot{f}}}$, for any $f\in L^2(\Sigma^2)$. Hence the equality $\|S_f\|^2=\|f\|^2$ reads
$\text{tr}(S_fS_{\overline{\dot{f}}})=\int f\overline{f}$. 
By polarization, we obtain that for all $f,h\in L^2(\Sigma^2)$, 
$\text{tr}(S_fS_h)=\int f\dot{h}$.
Consequently,
\begin{align} \label{egalite dualite}
\text{tr}\left(S_{\varphi f}S_h\right)=\int_{\Sigma^2}
\varphi(s,t)f(s,t)h(t,s)dsdt=\text{tr}\left(S_fS_{\widecheck{\varphi} h}\right),
\qquad f,h\in L^2(\Sigma^2).
\end{align}
This yields $4\Rightarrow 5\Rightarrow 2$.

By definition, $M_\varphi(S^2)\subset S^2\subset K$,
hence $1\Leftrightarrow 4$ follows from the definition.
Likewise we have $2\Rightarrow 4$ by taking 
restriction. This concludes the proof of the equivalence
of the five properties. The rest of the statement follows from 
the above arguments.
\end{proof}

Schur multipliers on $B(L^2(\Sigma))$ are characterized as follows.
We refer to \cite[theorem 3.3]{Spronk} for this statement.
 
\begin{theorem} \label{theoreme multi Schur alpha beta}
A function $\varphi\in L^\infty(\Sigma^2)$
is a Schur multiplier on $B(L^2(\Sigma))$ if and only if there exist
a Hilbert space $H$ and two functions 
$\alpha$ and $\beta$ in $L^\infty(\Sigma,H)$ such that
\begin{align}\label{equa alpha beta}
\varphi(s,t)=\left\langle\alpha(s),\beta(t)\right\rangle_H\text{ a.e. on }\Sigma\times\Sigma. 
\end{align}
\end{theorem}

In the rest of this section, 
we assume that $\varphi\in L^\infty(\Sigma^2)$ is a Schur multiplier on $B(L^2(\Sigma))$. The argument in the proof of 
Lemma \ref{th equivalent schur multipliers} shows that the Hilbertian
adjoint of 
$M_\varphi:S^2(L^2(\Sigma))\to S^2(L^2(\Sigma))$ is equal to $M_{\overline{\varphi}}$. We can therefore characterize 
the self-adjointness of a Schur multiplier as follows (see \cite[proposition 4.2]{A3}).

\begin{prop}
The operator $M_\varphi:S^2(L^2(\Sigma))\to S^2(L^2(\Sigma))$ 
is self-adjoint if and only if $\varphi$ is real-valued.
\end{prop}

As a complement to Theorem \ref{theoreme multi Schur alpha beta},
we mention the following classical result on Schur multipliers.
We refer to \cite[p. 61]{Pisier} for the notion of complete positivity.

\begin{theorem}\label{th equivalence positif completement positif}
The following are equivalent:
\begin{enumerate}
\item $M_\varphi$ is positive;
\item $M_\varphi$ is completely positive;
\item there exist
a Hilbert space $H$ and a function 
$\alpha\in L^\infty(\Sigma,H)$ such that 
\begin{align*}
\varphi(s,t)=\left\langle\alpha(s),
\alpha(t)\right\rangle_H\text{ a.e. on }\Sigma\times\Sigma. 
\end{align*}
\end{enumerate}
\end{theorem}

This theorem can be proved 
in two steps. The first one consists in proving it
in the finite dimensional case. This is done e.g.
in \cite[theorem 3.7]{Paulsen}. The second step consists in 
approximating from the finite dimensional case. In  \cite[Theorem 1.7]{Lafforgue}, the authors prove Theorem \ref{theoreme multi Schur alpha beta} using the finite dimensional case. Their
argument shows as well that Theorem \ref{th equivalence positif completement positif} can be deduced from the finite dimensional case. Our work in Section 4 will provide a direct proof of 
Theorem \ref{theoreme multi Schur alpha beta}
(see the proof of Theorem \ref{th equivalence positif completement positif},  
in particularly \eqref{eq juste alpha}).

The following result is apparently new.	
\begin{theorem}\label{th unital et alpha beta}
Let $H$ and $\alpha, \beta\in L^\infty (\Sigma,H)$ 
such that \eqref{equa alpha beta} holds true. Then 
\[ M_\varphi^\infty\text{ is unital }  \iff 
\left\langle \alpha(t),\beta(t)\right\rangle_H=1
\text{ a.e. on }\Sigma.\]
\end{theorem}

\begin{proof}
By linearity and density, we have:
\[ M^\infty_\varphi(I)=I\iff\forall f,h,\in L^2(\Sigma)\text{, } \left\langle M_\varphi^\infty(I),f\otimes h\right\rangle=\left\langle I,f\otimes h\right\rangle  \]
Hence by Lemma \ref{th equivalent schur multipliers},
\begin{align*}
M_\varphi^\infty(I)=I&\iff \forall f,h\in L^2(\Sigma)\text{, } \left\langle I,M_{\widecheck{\varphi}}^1(f\otimes h)\right\rangle  
=\left\langle I,f\otimes h\right\rangle  \\
&\iff \forall f,h\in L^2(\Sigma)
\text{, tr}\left(M_{\widecheck{\varphi}}^1(f\otimes h)\right)=\int_\Sigma f(t)h(t)dt.
\end{align*}

We can suppose that $H$ separable (and also that dim$(H)$ is infinite),
thanks to Pettis's measurability theorem (see \cite{Diesel}).
Thus there exists a Hilbertian basis $(e_n)_{n\geq 1}$ of $H$.
We denote, for all $n$, $\alpha_n(s)=\left\langle\alpha(s),e_n\right\rangle_H$ 
and $\beta_n(t)=\left\langle e_n,\beta(t)\right\rangle_H$ for a.e. $s\in\Sigma$ and $t\in\Sigma$. 
We have
$$
\left\langle \alpha(s),\beta(t)\right\rangle_H  =\sum\limits_n \alpha_n(s)\beta_n(t)\quad
a.e.\ \text{on}\ \Sigma\times\Sigma.
$$

Let $f,h\in L^2(\Sigma)$, we will prove that:
\begin{align}\label{eq sum alphaf et betah fini}
\sum_{n=1}^\infty\|\alpha_n h\|_2^2<\infty\text{ and }\sum_{n=1}^\infty\|\beta_n f\|_2^2<\infty
\end{align}
First,
\begin{align*}
\sum_{n=1}^\infty \|\alpha_n h\|^2_2 & = 
\sum_{n=1}^\infty\int_\Sigma |\alpha_n(s)|^2 |h(s)|^2ds\,
=\int_\Sigma\sum_{n=1}^\infty |\alpha_n(s)|^2 |h(s)|^2ds\\
&= \int_\Sigma\|\alpha(s)\|^2_H|h(s)|^2ds\,
\leq \|\alpha\|_\infty^2\|h\|_2^2<\infty.
\end{align*}
Similarly, we have
$\sum \|\beta_n f \|_2^2<\infty$.
By the Cauchy-Schwarz inequality, we deduce
\begin{align}\label{eq sumint alpha f beta h}
&\sum_{n=1}^\infty\int_\Sigma |\alpha_n(t)||h(t)||\beta_n(t)||f(t)|dt<\infty.
\end{align}
Moreover using (\ref{eq sum alphaf et betah fini}) again,
and (\ref{S1}), we may consider
$\sum\limits_{n=1}^\infty \beta_n f\otimes \alpha_n h$  in $S^1(L^2(\Sigma))$ and we have
\begin{align*}
\text{tr}\left(\sum\limits_{n=1}^\infty\beta_n f\otimes \alpha_n h\right)
&=\sum\limits_{n=1}^\infty\text{tr}\left(\beta_n f\otimes \alpha_n h\right)\\
&=\sum\limits_{n=1}^\infty\int_\Sigma\alpha_n(t) h(t)  \beta_n(t) f(t)dt\\
&=\int_\Sigma\sum\limits_{n=1}^\infty\alpha_n(t) h(t)  \beta_n(t) f(t)dt,
\ \text{by \eqref{eq sumint alpha f beta h}},\\
&=\int_\Sigma \langle \alpha (t),\beta(t)\rangle_H f(t)h(t)dt.
\end{align*}
Let us now show that
\begin{align*}
M_{\widecheck{\varphi}}^1(S_{f\otimes h})=\sum_{n=1}^\infty \beta_n f\otimes \alpha_n h.
\end{align*}
Let $v_1,v_2\in L^2(\Sigma)$. According to \eqref{eq sum alphaf et betah fini}, we have
\begin{align}\label{eq alphasfsbetatht fini}
\sum_n\int_{\Sigma^2}|\alpha_n(t)h(t)v_1(t) \beta_n(s)f(s)v_2(s)\vert
dtds<+\infty.
\end{align}
Then  we have (using the usual duality pairing $\langle\,\cdotp,\,\cdotp\rangle$ between $L^2(\Sigma)$ and itself)
\begin{align*}
\langle M_{\widecheck{\varphi}}^1(f\otimes h)v_1,v_2\rangle
&=\langle S_{\widecheck{\varphi}(f\otimes h)}v_1,v_2\rangle\\
&=\int_{\Sigma^2} \varphi(t,s)f(s)h(t)v_1(t)v_2(s)dtds\\
&=\int_{\Sigma^2} \langle\alpha(t),\beta(s)\rangle_H f(s)h(t)v_1(t)v_2(s)dtds\\
&=\int_{\Sigma^2} \sum_{n=1}^\infty \alpha_n(t)\beta_n(s) f(s)h(t)v_1(t)v_2(s)dtds\\
&=\sum_{n=1}^\infty \int_{\Sigma^2} \alpha_n(t)\beta_n(s) f(s)h(t)v_1(t)v_2(s)dtds, 
\text{ by \eqref{eq alphasfsbetatht fini}},\\
&=\sum_{n=1}^\infty \langle \alpha_n h, v_1 \rangle \langle \beta_n f,v_2\rangle\\
&=\left\langle \left( \sum_{n=1}^\infty \beta_n f\otimes \alpha_n h\right) v_1,v_2\right\rangle
\end{align*}
This proves the announced identity. We deduce that
\begin{align*}
\text{tr}(M^1_{\widecheck{\varphi}}(S_{f\otimes h}))=\int_\Sigma f(t)h(t) \left\langle \alpha(t),\beta(t)\right\rangle_H  dt.
\end{align*}
Finally,
\begin{align*}
M_\varphi^\infty\text{ is unital }&\iff \forall f,h\in L^2(\Sigma)\text{, }\int_\Sigma f(t)h(t)\left\langle \alpha (t),\beta (t)\right\rangle_H  dt=\int_\Sigma f(t)h(t)dt\\
&\iff \left\langle\alpha(t),\beta(t)\right\rangle_H  =1\ \text{a.e. on}\ \Sigma. 
\end{align*}
\end{proof}

\section{Dilatation of Schur multiplier}

The goal of the section is to prove Theorem (\textbf{A}), stated
in the Introduction.

Throughout we let $\Sigma$ be a $\sigma$-finite measure space,
and we let 
$\varphi\in L^\infty(\Sigma^2)$ such that $M_\varphi$ is a self-adjoint, unital, positive Schur multiplier on $B(L^2(\Sigma))$. We define a bilinear, symmetric  map $V: L^1_\R(\Sigma)\times L^1_\R(\Sigma)\to \R$ by
\begin{align*}
V(f,h)=\int_{\Sigma^2}\varphi(s,t)h(t)f(s)dtds,\qquad 
(f,h)\in L^1_\R(\Sigma)\times L^1_\R(\Sigma).
\end{align*} 
We claim that $V$ is positive. To prove it, fix some $f\in L^1_\R(\Sigma)\cap L^2_\R(\Sigma)$.
For all $E\subset \Sigma$ such that $\vert E\vert<\infty$, we have $\chi_E\otimes\chi_E\geq 0$. Hence
by the positivity of $M_\varphi$, we have 
\begin{align*}
\langle[ M_\varphi(\chi_E\otimes\chi_E)](f), f\rangle\geq 0.
\end{align*}
Equivalently,
\begin{align*}
\int_{E\times E}\varphi(s,t)f(t)f(s)dtds\geq 0.
\end{align*}
Passing to the supremum over $E$, we deduce
\begin{align*}
\int_{\Sigma\times \Sigma}\varphi(s,t)f(t)f(s)dtds\geq 0.
\end{align*}
Since $L^1_\R(\Sigma)\cap L^2_\R(\Sigma)$ is dense in 
$L^1_\R(\Sigma)$, this proves the claim.

Let $K_V\subset L^1_\R(\Sigma)$ be the kernel of the seminorm
$V(f,f)^\frac12$. Equipped with the resulting norm, the quotient
$L^1_\R(\Sigma)/K_V$ is a real pre-Hilbert space. We let 
$\mathbb{H}$ denote its completion. This is a real Hilbert space.
Let $N$ denote the tracial 
von Neumann algebra $(\Gamma_{-1}(\mathbb{H});\tau)$ (see section 2).

In the sequel, wherever $h\in L^1_\R(\Sigma)$, 
we let $\dot{h}\in\mathbb{H}$ denote its class. 
We let $T: L^1(\Sigma)\to N$  be the unique linear
map such that  for all $h\in L^1_\R(\Sigma)$, 
$$
T(h)=\omega(\dot{h}).
$$
We let $F\in L^\infty_\sigma(\Sigma, N)$
and $d_0\in L^\infty(\Sigma)\overline{\otimes}N$ be 
associated with $T\in B(L^1(\Sigma), N)$ through the identifications
$B(L^1(\Sigma),N)\simeq L^\infty_\sigma(\Sigma,N)$ 
and $B(L^1(\Sigma),N)\simeq L^\infty(\Sigma)\overline{\otimes}N$ 
provided by (\ref{th identification lsigmainfty ...}).

\begin{lemme}\label{Norm1}
We have $\|T\|\leq 1$, $\|F\|\leq 1$ and $\|d_0\|\leq 1$.
\end{lemme}

\begin{proof}
We recall that $M_\varphi$ is positive and unital. By the
Russo-Dye Theorem, this implies that $\|M^\infty_\varphi\|=1$. 
It follows from Lemma \ref{th equivalent schur multipliers} that 
$\|M^1_\varphi\|=1$. Hence by interpolation, we have
$\Vert M_\varphi  :
S^2(L^2(\Sigma))\to S^2(L^2(\Sigma))\Vert\leq 1$. Consequently,
$\|\varphi\|_\infty\leq 1$ thanks to \eqref{lem norm phi}.

For all $f\in L_\R^1(\Sigma)$, we have
\begin{align*}
\|\dot{f}\|_\mathbb{H}^2=\int_{\Sigma^2}\varphi(s,t)f(t)f(s)dtds\leq \|\varphi\|_\infty\|f\|_1^2\leq \|f\|^2_1.
\end{align*}
So we have, thanks to \eqref{eq équilité norme},
\begin{align}\label{eq majoration norme f}
\|\omega(\dot{f})\|=\|\dot{f}\|_\mathbb{H}\leq \|f\|_1.
\end{align}
Consider $h=\sum\limits_{i=1}^n\alpha_i\chi_{E_i}$, where 
$E_1,\dots E_n$  are pairwise disjoint elements of finite 
measure and $\alpha_1,\ldots,\alpha_n\in{\mathbb C}$. 
Then we have $\|h\|_1=\sum|\alpha_i| \vert E_i\vert$. 
On the other hand, 
$$
T(h)=\sum \alpha_i T\left(\chi_{E_i}\right)
= \sum \alpha_i w\left(\dot{\chi_{E_i}}\right).
$$
Hence
\begin{align*}
\|T(h)\|_1 \leq \sum 
|\alpha_i|\left\|w\left(\dot{\chi_{E_i}}\right)\right\| 
\leq \sum|\alpha_i|\|\chi_{E_i}\|_1
=\sum| \alpha_i|\vert E_i\vert =\|h\|_1,
\end{align*}
by \eqref{eq majoration norme f}.
This shows that $\|T\|\leq 1$, and hence $\|F\|\leq 1$ and $\|d_0\|\leq 1$.
\end{proof}

The following lemma provides a link between $F$ and $\varphi$.

\begin{lemme}\label{varphi=tau FF}
For almost every $(s,t)\in \Sigma^2$, $\tau(\tilde{F\times F}(s,t))=\varphi(s,t)$.
\end{lemme}

\begin{proof}
We use the equation \eqref{eq phitildetimes phi} with $\eta=\tau$.
We obtain for all $f,h\in L^1_\R(\Sigma)$,
\begin{align*}
\int_{\Sigma^2}\tau(\tilde{F\times F}(s,t))f(s)h(t)dtds&
=\tau(T(f)T(h))\\
&=\tau(\omega(\dot{f})\omega(\dot{h}))\\
&=\int_{\Sigma^2}\varphi(s,t)f(s)h(t)dtds,
\end{align*}
by (\ref{Trace}). The expected equality follows at once.
\end{proof}

\begin{lemme}
We have $d_0^*=d_0$ and $d_0^2=1$.
\end{lemme}

\begin{proof}
We let $\tilde{\tau}=\int\,\cdotp\overline{\otimes}\tau$ denote the usual trace on $L^\infty(\Sigma)\overline{\otimes}N$.
For any $f\in L^1_\R(\Sigma)$ and $\eta\in N_*\simeq L^1(N)$, we have
\begin{align*}
\langle d_0^*,f\otimes \eta\rangle
&=\tilde{\tau}(d_0^*(f\otimes \eta))= 
\overline{\tilde{\tau}(d_0(f\otimes \eta^*))}\\
&=\overline{\tau(T(f)\eta^*)}
=\overline{\tau (w(\dot{f}) \eta^*)}\\
&=\tau (w(\dot{f}) \eta)
= \tau(T(f)\eta)
=\langle d_0,f\otimes \eta\rangle.
\end{align*}
By linearity and density, this implies that $d_0^*=d_0$.

By assumption, $(\Sigma,\mu)$ is $\sigma$-finite. Thus
there exists a positive function $v\in L^1(\Sigma)$ such that $\|v\|_1=1$. Changing
$d\mu$ into $vd\mu$, we may therefore assume
that $\mu$ is a probability measure.
Since $\tau$ is normalized, the trace
$\tilde{\tau}$ is normalized as well.

Let $j:N \to L^2(N )$ be the natural embedding and 
let $T_2=j\circ T:L^1(\Sigma)\to L^2(N )$. Since $L^2(N )$ is a Hilbert space, 
it has the Radon-Nikodym property (see \cite[corollary IV.1.4]{Diesel}). 
Hence there exists $\gamma\in L^\infty(\Sigma,L^2(N ))$ such that for 
all $h\in L^1(\Sigma)$:
\begin{align*}
T_2(h)=\int_\Sigma h(t)\gamma(t)dt.
\end{align*} 
For any $f,h\in L^1_\R(\Sigma)$, we have
\begin{align*}
\int_{\Sigma^2} \varphi(s,t)f(s)h(t)dtds
&=\tau (\omega(\dot{f})\omega(\dot{h})),\quad\text{by}\,(\ref{Trace}),\\
&=\tau (\omega(\dot{f})\omega(\dot{h})^*)\\
&=\langle j(\omega(\dot{f})),j(\omega(\dot{h}))\rangle_{L^2(N)}\\
&=\langle T_2(f),T_2(h)\rangle_{L^2(N)}\\
&=\left\langle\int_\Sigma f(s)\gamma(s)ds,\int_\Sigma h(t)\gamma(t)dt\right\rangle_{L^2(N)}\\
&=\int_{\Sigma^2} \langle \gamma(s),\gamma(t)\rangle_{L^2(N)} f(s)h(t)dtds.
\end{align*}
This implies that for almost every $(s,t)\in \Sigma^2$, 
\begin{align}\label{eq juste alpha}
\langle \gamma(s),\gamma(t)\rangle_{L^2(N)}=\varphi(s,t).
\end{align}
We obtained a factorization of $\varphi$ as in 
Theorem \ref{theoreme multi Schur alpha beta}, with $\alpha=\beta=\gamma$.
Now we can use Theorem \ref{th unital et alpha beta} 
and we obtain that for almost every $t\in\Sigma$,
\begin{align*}
\langle \gamma(t),\gamma(t)\rangle_{L^2(N)}=1
\end{align*}
Since the measure on $\Sigma$ is normalized, this yields
\begin{align*}
\int_\Sigma \|\gamma(t)\|^2_{L^2(N )}dt=1.
\end{align*}
Since $L^\infty(\Sigma)\overline{\otimes}N$ is normalized we have, using
(\ref{Lp}), a contractive inclusion 
$$
L^\infty(\Sigma)\overline{\otimes}N\subset L^2(\Sigma,L^2(N)).
$$
By construction, $\gamma\in L^2(\Sigma,L^2(N))$ corresponds to $d_0$.
Thus we have proved that
$$
\tilde{\tau}(d_0^*d_0)=1.
$$
Since $\tilde{\tau}$ is normalized,
this shows that $\tilde{\tau}(1-d_0^*d_0)=0$. However
by Lemma \ref{Norm1}, $\Vert d_0\Vert\leq 1$ hence $1-d_0^*d_0\geq 0$. 
Since $\tilde{\tau}$ is faithful, we deduce that $1-d_0^*d_0 = 0$, that is,
$d_0^2=1$.
\end{proof}

\begin{rmq}
The identity \eqref{eq juste alpha} provides a
new proof of Theorem 
\ref{th equivalence positif completement positif}.
\end{rmq}

We are now ready to introduce the maps $U,J$ providing 
the absolute dilation of $M_\varphi$. We let $N^\infty$ denote
the infinite von Neumann tensor product $\overline{\otimes}_\Z N$ and we 
let $\tau^\infty$ denote the normal faithful finite trace on $N^\infty$ 
(see \cite{Takesaki3}).
We set 
$$
F^\infty=\cdots \otimes 1_N\otimes F\otimes 1_N\cdots \in L_\sigma^\infty(\Sigma, N^\infty),
$$
where $F$ is in 0 position.  We denote by $T^\infty\in B(L^1(\Sigma),N^\infty)$ 
the element associated with $F^\infty$ through the identification between $L_\sigma^\infty(\Sigma, N^\infty)$ and 
$B(L^1(\Sigma),N^\infty)$. We also let $d\in L^\infty(\Sigma)\overline{\otimes}N^\infty$ be associated with $F^\infty$ through the identification between
$L_\sigma^\infty(\Sigma,N^\infty)$ and $L^\infty(\Sigma)\overline{\otimes}N^\infty$.
We know from above that
\begin{align*}
\|d\|_\infty\leq 1\text{, }d^*=d\text{ and }d^2=1.
\end{align*}
In the sequel, we regard $L^\infty(\Sigma)$ as a 
von Neumann subalgebra of $B(L^2(\Sigma))$
by identifying any $\phi\in L^\infty(\Sigma)$ with the multiplication operator,
$h\mapsto\phi h$ for $h\in L^2(\Sigma)$. Thus we have
$$
L^\infty_\sigma(\Sigma,N^\infty)\simeq 
L^\infty(\Sigma)\overline{\otimes}N^\infty\subset B(L^2(\Sigma))\overline{\otimes}N^\infty.
$$
We will see $d$ as an element of $B(L^2(\Sigma))\overline{\otimes}N^\infty$.

We recall that tr denotes the trace on $B(L^2(\Sigma))$, we let 
$N'=B(L^2(\Sigma))\overline{\otimes}N^\infty$
and we let
$\tau_{N'}=\text{tr }\overline{\otimes}\tau^\infty$ be the natural 
semifinite normal faithful trace on $N'$. 
We define 
$$
J: B(L^2(\Sigma))\longrightarrow B(L^2(\Sigma))\overline{\otimes}N^\infty,
\qquad J(x)=x\otimes 1_{N^\infty}.
$$
This is a trace preserving one-to-one $\ast$-homomorphism.
We let $\E : N'\to B(L^2(\Sigma))$ denote the conditional 
expectation associated with $J$.

We introduce the right shift $\mathcal{S} : N^\infty\to N^\infty$. This is
a normal, trace preserving $\ast$-automorphism such that 
for all $(x_n)_{n\in \Z}\subset N$,
\begin{align*}
\mathcal{S}(\cdots\otimes x_0\otimes x_1\otimes x_2\otimes \cdots)=\cdots\otimes x_{-1}\otimes x_0\otimes x_1\otimes \cdots.
\end{align*}
We define 
$$
U : B(L^2(\Sigma))\overline{\otimes}N^\infty\longrightarrow 
B(L^2(\Sigma))\overline{\otimes}N^\infty
$$
by
$$
U(y) = d((Id\otimes\mathcal{S})(y))d,\qquad y\in B(L^2(\Sigma))\overline{\otimes}N^\infty.
$$
Since $d$ is a self-adoint symmetry, $U$ is a normal, trace preserving
$\ast$-automorphism.

We apply Lemma \ref{prop existence gamma} with $N^\infty$ and we let
$$
\Gamma : L^2_\sigma(\Sigma^2, N^\infty)\longrightarrow B(L^2(\Sigma))\overline{\otimes}N^\infty
$$
be the resulting $w^*$-continuous contraction.

\begin{lemme}\label{prop R}
The mapping $R:L^2_\sigma(\Sigma^2,N^\infty)\to L^2_\sigma(\Sigma,N^\infty)$ defined by $R(F)=\mathcal{S}\circ F$ is  $w^*$- continuous. In addition,
we have
\begin{align}\label{formule avec R et Gamma}
(Id\otimes\mathcal{S})\Gamma=\Gamma R.
\end{align}
\end{lemme}

\begin{proof}
We apply Lemma \ref{lem RcircF} with $V=\mathcal{S}$ to prove the first point. 
For all $\theta \in L^2(\Sigma)$, and $y\in N^\infty$, we have
\begin{align*}
(Id\otimes \mathcal{S} )\Gamma(\theta\otimes y)= S_\theta\otimes \mathcal{S}(y)=\Gamma(\theta\otimes\mathcal{S} (y))=\Gamma(R(\theta\otimes y)).
\end{align*}
We deduce the equality $(Id\otimes\mathcal{S})\Gamma=\Gamma R$ on 
$L^2_\sigma(\Sigma^2,N^\infty)$, by linearity, $w^*$-continuity of the maps and $w^*$-density of $L^2(\Sigma^2)\otimes N^\infty$ in $L^2_\sigma(\Sigma^2,N^\infty)$.
\end{proof}

For any $k\in \N$ we set $N^k=\overset{\text{k times}}{\overbrace{N\overline{\otimes}\cdots\overline{\otimes}N}}$. 
Then we let $L_k:N^k\to N^\infty$ be the unique $w^*$-continuous 
$\ast$-homomorphism such that:
\begin{align*}
L_k(z_1\otimes\cdots\otimes z_k)=\cdots\otimes 1_N\otimes z_1\otimes\dots\otimes z_k\otimes 1_N\cdots,\qquad z_1,\ldots,z_k\in N,
\end{align*}
where $z_j$ is at position $j$ for all $1\leq j\leq k$.
For the rest of the paper, we identify $N^k$ with its image $L_k(N^k)$.

\begin{lemme}\label{prop fonctions prefaiblement continues}

\  

\begin{enumerate}
\item Let $f\in L^\infty_\sigma(\Sigma,N)$, $k\in \N$,
and let $A:L^\infty_{\sigma}(\Sigma,N^k)\to L^\infty_\sigma(\Sigma,N^\infty)$ be 
defined by 
\begin{align*}
\begin{array}{ccccccccc}
A(g) =\cdots1_N\otimes 1_N\otimes &f&\otimes &g&\otimes& 1_N&\cdots\\
&&&\uparrow&&\uparrow&\\
&&&1&&k+1&
\end{array}
\end{align*}
Then $A$ is $w^*$-continuous.
\item Let $\Psi\in L^\infty_\sigma(\Sigma,N^\infty)$
and let $B:L^\infty_{\sigma}(\Sigma,N^\infty)\to L^\infty_\sigma(\Sigma^2,N^\infty)$ 
be defined by 
$$
B(\Phi) = \tilde{\Phi\times\Psi}.
$$
Then $B$ is $w^*$-continuous.
\end{enumerate}
\end{lemme}

\begin{proof}
In the definition of $A$, the $1_N$ play no role. So we may consider $A:L^\infty_{\sigma}(\Sigma,N^k)\to L^\infty_\sigma(\Sigma,N\overline{\otimes}N^k)$ 
such that $A(h)=f\otimes h$, instead of the map given in the statement. 
To prove that $A$ is $w^*$-continuous, it suffices to show that for any 
bounded net $(h_i)$ which converges to $h$ in the $w^*$-topology of 
$L^\infty_\sigma(\Sigma, N^k)$ and for any $V\in (L_\sigma^\infty(\Sigma,N\overline{\otimes}N^k))_*=L^1(\Sigma,(N\overline{\otimes}N^k)_*)$, we have
\begin{align}\label{eq A est bien precontinue}
\langle A(h_i),V\rangle\to\langle A(h),V\rangle.
\end{align}
We know that $L^1(\Sigma)\otimes N_*\otimes N^k_*$ is dense in $L^1(\Sigma,(N\overline{\otimes}N^k))_*$ (see the beginning of section 2).
Let $(h_i)\subset L^\infty_\sigma(\Sigma,N^k)$ be a bounded net which $w^*$-converges to 
$h$. Then for all $v\in L^1(\Sigma)$ and $\mu\in N^k_*$, we have:
\begin{align}\label{eq convergence gi}
\langle h_i,v\otimes \mu\rangle\to\langle h, v\otimes \mu\rangle
\end{align}
We note that for all $\eta\in N_*$, $\mu\in N^k_*$ and $l\in L^1(\Sigma)$:
\begin{align*}
\langle A(h_i),l\otimes \eta\otimes \mu\rangle=\int_\Sigma l(t)\langle f(t),\eta\rangle\langle h_i(t),\mu\rangle dt.
\end{align*}
Using \eqref{eq convergence gi} with $v(t)=l(t)\langle f(t),\eta\rangle$, we derive 
\begin{align*}
\langle A(h_i),l\otimes \eta\otimes \mu\rangle\to\langle A(h),l\otimes \eta\otimes \mu\rangle.
\end{align*}
By linearity and density, and by the boundedness of $(h_i)$, we obtain \eqref{eq A est bien precontinue}.

Now we prove the $w^*$- continuity of $B$. 
We let  
$$
\tilde{B}: B(L^1(\Sigma),N^\infty)\to B(L^1(\Sigma^2),N^\infty)
$$
be the operator corresponding to $B$ if we use the identifications $B(L^1(\Sigma);N^\infty)=
L^\infty_\sigma(\Sigma,N^\infty)$
and $B(L^1(\Sigma^2);N^\infty)=L^\infty_\sigma(\Sigma^2,N^\infty)$.
We will prove that $\tilde{B}$ is $w^*$-continuous.
Let $S\in B(L^1(\Sigma);N^\infty)$ 
be corresponding to $\Psi$. Then 
for all $T\in B(L^1(\Sigma),N^\infty)$ and for all $f,h\in L^1(\Sigma)$, we have
\begin{align*}
\tilde{B}(T)(f\otimes h)=T(f)S(h)
\end{align*}
We know that $L^1(\Sigma)\otimes L^1(\Sigma)\otimes N^\infty_*$ is dense in $L^1(\Sigma^2;N^\infty_*)=B(L^1(\Sigma^2);N^\infty)_*$. 
For all $f,h\in L^1(\Sigma)$ and $\eta \in N^\infty_*$, we have
\begin{align}\label{eq B tilde}
\langle \tilde{B}(T),f\otimes h\otimes \eta\rangle=\langle T(f)S(h),\eta\rangle=\langle T(f),S(h)\eta\rangle=\langle T,f\otimes S(h)\eta\rangle.
\end{align}
Let $ (T_i)\subset B(L^1(\Sigma);N^\infty)$ be a bounded net which converges to $T\in B(L^1(\Sigma);N^\infty)$ in the $w^*$-topology. We have:
\begin{align*}
\langle T_i,f\otimes S(h)\eta\rangle\to \langle T,f\otimes S(h)\eta\rangle
\end{align*}
We therefore obtain, thanks to \eqref{eq B tilde}, that 
\begin{align*}
\langle \tilde{B}(T_i),f\otimes h\otimes \eta\rangle\to \langle \tilde{B}(T),f\otimes h\otimes \eta\rangle.
\end{align*}
By linearity, by density and 
by the boundedness of $(T_i)$, 
we obtain that for all $V\in L^1(\Sigma,N^\infty_*)$:
\begin{align*}
\langle \tilde{B}(T_i),V\rangle\to \langle \tilde{B}(T),V\rangle.
\end{align*}
This shows the $w^*$-continuity of $\tilde{B}$, and hence that of $B$.
\end{proof}

\begin{rmq}\label{remarque prefaiblecontinuité}
Let $\Psi\in L_\sigma^\infty(\Sigma,N^\infty)$, $\theta\in L^2(\Sigma^2)$ and $y\in N^\infty$.
We deduce from the previous lemma the $w^*$-continuity of the map $\Phi\mapsto  \theta(\tilde{\Phi\times\Psi})y$ from $L_\sigma^\infty(\Sigma,N^\infty)$ into $L^2_\sigma(\Sigma^2,N^\infty)$.
\end{rmq}

The following lemma explains how to swap $d$ and $\Gamma$.

\begin{lemme}\label{lemme d et gamma}
For all $\theta\in L^2(\Sigma^2)$ and for all $y=\dots\otimes y_{-1}\otimes 1_N\otimes y_1\dots\in N^\infty$, we have:
\begin{align}\label{formule d et Gamma}
d\Gamma(\theta\otimes y)d=\Gamma\left((s,t)\mapsto \theta(s,t)[\cdots\otimes y_{-1}\otimes \tilde{F\times F}(s,t)\otimes y_1\otimes\cdots]\right)
\end{align}
\end{lemme}

\begin{proof} Recall ${\mathcal E}$ given by (\ref{E}).
Instead of $d$, we first consider approximations by fnite sums
\begin{align*}
\phi= \sum_{i}\chi_{E_i}\otimes\dots\otimes 1_N\otimes  m_i\otimes 1_N\otimes\cdots\ \in {\mathcal E}\otimes N^\infty
\end{align*}
and
\begin{align*}
\phi'= \sum_{j}\chi_{E_j'}\otimes\dots\otimes 1_N\otimes m_j'\otimes 1_N\otimes\cdots\ \in  {\mathcal E}\otimes N^\infty.
\end{align*}  
We have:
\begin{align*}
\phi\Gamma (\theta\otimes y)\phi'&=\phi(S_\theta\otimes y)\phi'\\
&=\sum_{i,j}\chi_{E_i}.S_\theta.\chi_{E_j'}\otimes \dots\otimes y_{-1}\otimes m_i m_j'\otimes y_1\otimes\cdots\\
&=\sum_{i,j}\Gamma\left((s,t)\mapsto \chi_{E_i}(s)\theta(s,t)\chi_{E_j'}(t)[\dots\otimes y_{-1}\otimes m_i m_j'\otimes y_1\otimes\cdots]\right)\\
&=\Gamma\left((s,t)\mapsto \theta(s,t)\Bigl[\dots\otimes y_{-1}\otimes
\Bigl(\sum_i\chi_{E_i}(s) m_i\Bigr)\Bigl(\sum_{j}\chi_{E_j'}(t)m_j'\Bigr)\otimes y_1\otimes\cdots\Bigr]\right)\\
&=\Gamma\left((s,t)\mapsto \theta(s,t)\phi(s)\phi'(t)y\right).
\end{align*}
Here $\phi$ and $\phi'$ are measurable, hence
we have $\phi(s)\phi'(t)= \tilde{\phi\times\phi'}(s,t)$, by 
Remark \ref{rmq produit tilde simple}. Hence the preceding equality gives:
\begin{align*}
\phi\Gamma (\theta\otimes y)\phi'=\Gamma\bigl((s,t)\mapsto \theta(s,t)\tilde{\phi\times\phi'}(s,t)y\bigr)
\end{align*}
In this identity, the maps $\phi$ and $\phi'$ are 
regarded as elements of $B(L^2(\Sigma))\overline{\otimes}N^\infty$
on the left hand-side, and as 
elements of $L^\infty(\Sigma,N^\infty)$ on the right hand-side.
Using the facts that $\mathcal E\otimes N^\infty$ is $w^*$-dense in $L^\infty_\sigma(\Sigma,N^\infty)$, that $\Gamma$ is $w^*$- continuous, as well as
Lemma \ref{prop fonctions prefaiblement continues} and Remark \ref{remarque prefaiblecontinuité}, we deduce that
\begin{align*}
\phi\Gamma(\theta\otimes y)d=\Gamma\left((s,t)\mapsto \theta(s,t)(\phi\tilde{\times} F^\infty)(s,t)y\right)
\end{align*}
and then
\begin{align*}
d\Gamma(\theta\otimes y)d=\Gamma\left((s,t)\mapsto \theta(s,t)(\tilde{F^\infty\times F^\infty})(s,t)y\right).
\end{align*}
The last thing to observe is that $F^\infty\tilde{\times}F^\infty =\dots \otimes 1_N\otimes F\tilde{\times} F\otimes 1_N\otimes\cdots$ . The latter 
is true by the uniqueness of the construction.
\end{proof}

\begin{lemme}\label{prop uk=gamma...}
For all $f,h\in L^2(\Sigma)$, and  for all integer $k\geq 0$,  we have:
\begin{align}
\begin{array}{ccccc}
U^kJ(S_{f\otimes h})=\Gamma\Bigl((s,t)\mapsto f(s)h(t)[\dots \otimes 1_N\otimes &F\tilde{\times}F(s,t)&\otimes\dots\otimes&F\tilde{\times} F(s,t)&\otimes1_N\otimes\cdots]\Bigr).\\
&\uparrow&&\uparrow&\\
&0&&(k-1)&
\end{array}
\end{align}
\end{lemme}
\begin{proof}
We prove this lemma by induction. The result is true for $k=0$ since
\begin{align*}
J(S_{f\otimes h})&=S_{f\otimes h}\otimes\dots\otimes 1_N\otimes\cdots\\
&=\Gamma ((f\otimes h)\otimes 1_{N^\infty})\\
&=\Gamma\bigl((s,t)\mapsto f(s)h(t)[\dots\otimes 1_N\otimes\cdots]\bigr).
\end{align*}
We now suppose that the result holds true for some $k\geq 0$. By
the $w^*$-density of $L^\infty(\Sigma^2)\otimes N$ into $L_\sigma^\infty(\Sigma^2,N)$,
we may write
\begin{align*}
&(s,t)\mapsto f(s)h(t)[\cdots\otimes 1_N\otimes F\tilde{\times}F(s,t)\otimes\cdots\otimes F\tilde{\times}F(s,t)\otimes 1_N\otimes\cdots]\\
&= w^*-\lim_i\sum_{j \text{ finite}}\theta_j^i\otimes\cdots\otimes 1_N\otimes m_{i,j}^0\otimes\cdots\otimes m_{i,j}^{k-1}\otimes 1_N\otimes\cdots,\\
\end{align*}
where $\theta_j^i\in L^\infty(\Sigma^2)$ and $m_{i,j}^l\in N$ for all $i,j,l$. 
We have:
\begin{align*}
&U^{k+1}(J(S_{f\otimes h}))=U(U^kJ(S_{f\otimes h}))\\
&= U(\Gamma((s,t)\mapsto f(s)h(t)[\cdots \otimes 1_N\otimes F\tilde{\times}F(s,t)\otimes\cdots\otimes F\tilde{\times}F(s,t)\otimes 1_N\otimes\cdots])\\
&=U\left(\Gamma\left( w^*-\lim_i\sum_{j \text{ finite}}\theta_j^i\otimes\cdots\otimes 1_N\otimes m_{i,j}^0\otimes\cdots\otimes m_{i,j}^{k-1}\otimes 1_N\otimes\cdots\right)\right)\\
&= w^*-\lim_i\sum_{j \text{ finite}} U\left(\Gamma\left(\theta_j^i\otimes\cdots\otimes 1_N\otimes m_{i,j}^0\otimes\cdots\otimes m_{i,j}^{k-1}\otimes 1_N\otimes\cdots\right)\right).
\end{align*}
By w$^*$-continuity of $\Gamma$ and $U$,
\begin{align*}
&= w^*-\lim_i\sum_{j \text{ finite}} d\left(((Id\otimes \mathcal{S}) (\Gamma(\theta_j^i\otimes\cdots\otimes 1_N\otimes m_{i,j}^0\otimes\cdots\otimes m_{i,j}^{k-1}\otimes 1_N\otimes\cdots))\right)d\\
&= w^*-\lim_i\sum_{j \text{ finite}} d(\Gamma(\theta_j^i\otimes \mathcal{S}(\cdots\otimes 1_N\otimes m_{i,j}^0\otimes\cdots\otimes m_{i,j}^{k-1}\otimes 1_N\otimes\cdots)))d.
\end{align*}
By Lemma \ref{lemme d et gamma}, the latter is equal to
\begin{align*}
&\begin{array}{ccccc}
w^*-\lim\limits_i \sum\limits_j \Gamma\bigl((s,t)\mapsto \theta^i_j(s,t)[\cdots\otimes 1_N\otimes& F\tilde{\times}F(s,t)&\otimes m_{i,j}^0\otimes \cdots\otimes& m_{i,j}^{k-1}&\otimes 1_N\otimes\cdots]\bigr)\\
&\uparrow&&\uparrow&\\
&0&&k&
\end{array}\\
\end{align*}
By Lemma \ref{prop fonctions prefaiblement continues}, (1), and 
the w$^*$-continuity of $\Gamma$, this is equal to
\begin{align*}
&\begin{array}{ccccccc}
\Gamma\bigl((s,t)\mapsto [\cdots\otimes &F\tilde{\times}F(s,t)&\otimes \Bigl(w^*-\lim_i\sum_{j \text{ finite}} \theta^i_j(s,t)m_{i,j}^0\otimes\cdots\otimes m_{i,j}^{k-1}\Bigr)\otimes &1_N&\cdots]\bigr)\\
&\uparrow&&\uparrow&\\
&0&&k+1&
\end{array}
\end{align*}
We let $\psi\in L^2_\sigma(\Sigma,N^\infty)$ be defined by \begin{align*}
\psi(s,t)=\cdots\otimes F\tilde{\times}F(s,t)\otimes 
\Bigl( w^*-\lim_i\sum_{j \text{ finite}} \theta^i_j(s,t)m_{i,j}^0\otimes\cdots\otimes m_{i,j}^{k-1}\Bigr)\otimes 1_N\cdots.
\end{align*}
Then we have obtained that
\begin{align*}
U^{k+1}(J(S_{f\otimes h}))= \Gamma(\psi)
\end{align*}
Recall the mapping 
$R: L^2(\Sigma^2,N^\infty)\to L^2(\Sigma^2,N^\infty)$ from Lemma \ref{prop R}. Then we have
\begin{align*}
\psi&= [\cdots\otimes F\tilde{\times}F\otimes ( w^*-\lim_i\sum_{j \text{ finite}} \theta^i_jm_{i,j}^0\otimes\cdots\otimes m_{i,j}^{k-1})\otimes 1_N\cdots].\\
&\begin{array}{ccccccc}
= [\cdots\otimes &F\tilde{\times}F&\otimes 1_N\cdots] w^*-\lim\limits_i\sum\limits_{j \text{ finite}} \theta^i_j[\cdots\otimes &1_N&\otimes m_{i,j}^0\otimes\cdots\otimes&m_{i,j}^{k-1}&\otimes\cdots].\\
&\uparrow&&\uparrow&&\uparrow&\\
&0&&0&&k&
\end{array}\\
&\begin{array}{ccccccc}
=[\cdots\otimes &F\tilde{\times} F(s,t)&\otimes 1_N\cdots] w^*-\lim\limits_iR(\sum\limits_{j \text{ finite}} \theta^i_j(s,t)[\cdots\otimes &m_{i,j}^0&\otimes\cdots\otimes& m_{i,j}^{k-1}&\otimes\cdots])\\
&\uparrow&&\uparrow&&\uparrow&\\
&0&&0&&(k-1)&
\end{array}\\
\end{align*}
By the $w^*$- continuity of $R$,
\begin{align*}
&=[\cdots\otimes F\tilde{\times} F\otimes 1_N\cdots]R( w^*-\lim_i\sum_{j \text{ finite}} \theta^i_j[\cdots\otimes m_{i,j}^0\otimes\cdots\otimes m_{i,j}^{k-1}\otimes\cdots])\\
&\begin{array}{ccccccc}
=[\cdots\otimes &F\tilde{\times}F&\otimes 1_N\cdots](f\otimes h)[\cdots\otimes &1_N&\otimes F\tilde{\times}F\otimes\cdots\otimes &F\tilde{\times}F&\otimes \cdots]\\
&\uparrow&&\uparrow&&\uparrow&\\
&0&&0&&k&
\end{array}\\
&\begin{array}{ccccc}
= f\otimes h[\cdots \otimes 1_N\otimes &F\tilde{\times}F&\otimes\cdots\otimes&F\tilde{\times}F&\otimes 1_N\otimes\cdots]).\\
&\uparrow&&\uparrow&\\
&0&&k&
\end{array}
\end{align*}
We obtain the result for $k+1$. This proves the lemma.
\end{proof}

We now conclude the proof of Theorem (\textbf{A}) by proving 
\begin{equation}
(M_\varphi^\infty)^k=\E U^kJ,\qquad k\geq 0.
\end{equation}

By linearity, density and duality, it suffices to prove that 
for all $f,h\in L^2(\Sigma)$ and for all $u,v\in L^2(\Sigma)$, we have 
\begin{align}\label{eq dualite dernier calcul}
\left\langle \E U^kJ(S_{f\otimes h}),u\otimes v\right\rangle=\left\langle M_\varphi^k(S_{f\otimes h}),u\otimes v\right\rangle.
\end{align}
Let $\gamma : \bigl(B(L^2(\Sigma))\overline{\otimes}N^\infty\bigr)_*
\to L^2(\Sigma^2,N^\infty_*)$ such that $\gamma^*=\Gamma$. 
We write
\begin{align*}
\left\langle \E U^kJ(S_{f\otimes h}),u\otimes v\right\rangle=\left\langle  U^kJ(S_{f\otimes h}),J_1(u\otimes v)\right\rangle.
\end{align*}
By lemma \ref{prop uk=gamma...}, the latter is equal to
\begin{align*}
&=\left\langle  \Gamma((s,t)\mapsto f(s)h(t)[\cdots\otimes F\tilde{\times}F(s,t)\otimes\cdots\otimes F\tilde{\times}F(s,t)\otimes 1_N\otimes\cdots]),J_1(u\otimes v)\right\rangle\\
&=\left\langle  (s,t)\mapsto f(s)h(t)[\cdots\otimes F\tilde{\times}F(s,t)\otimes\cdots\otimes F\tilde{\times}F(s,t)\otimes 1_N\otimes\cdots],\gamma(J_1(u\otimes v))\right\rangle\\
&=\left\langle  (s,t)\mapsto f(s)h(t)[\cdots\otimes F\tilde{\times}F(s,t)\otimes\cdots\otimes F\tilde{\times}F(s,t)\otimes 1_N\otimes\cdots],\gamma(u\otimes v\otimes 1_{N^\infty_*})\right\rangle\\
&=\left\langle  (s,t)\mapsto f(s)h(t)[\cdots\otimes F\tilde{\times}F(s,t)\otimes\cdots\otimes F\tilde{\times}F(s,t)\otimes 1_N\otimes\cdots],(s,t)\mapsto u(s)v(t) 1_{N^\infty_*})\right\rangle\\
&=\int_{\Sigma^2}\left\langle f(s)h(t)[\cdots\otimes F\tilde{\times}F(s,t)\otimes\cdots\otimes F\tilde{\times}F(s,t)\otimes 1_N\otimes\cdots], u(s)v(t) 1_{N^\infty_*})\right\rangle dsdt.\\
\end{align*}
Consequently,
\begin{align*}
\left\langle \E U^kJ(S_{f\otimes h}),u\otimes v\right\rangle &=\int_{\Sigma^2}f(s)h(t)u(s)v(t)\langle F\tilde{\times}F(s,t),1_{N_*}\rangle^kdsdt\\
&=\int_{\Sigma^2}f(s)h(t)u(s)v(t)\tau( F\tilde{\times}F(s,t))^kdsdt\\
&=\int_{\Sigma^2}f(s)h(t)u(s)v(t)\varphi(s,t)^kdsdt,
\text{ by Lemma \ref{varphi=tau FF}}\\
&=\left\langle M_\varphi^k(S_{f\otimes h}),u\otimes v\right\rangle
\end{align*}
This shows \eqref{eq dualite dernier calcul} and concludes the proof.

By Lemma \ref{lem absdila implique dila p}, we have the following
explicit consequence of Theorem (\textbf{A}).

\begin{theorem}
Let $\Sigma$ be a $\sigma$-finite measure space. 
Let $M_\varphi$ is a self-adjoint, unital, positive Schur multiplier on $B(L^2(\Sigma))$.
Then for all $1\leq p<\infty$, the map $(M_\varphi)_p:S^p(L^2(\Sigma))\to S^p(L^2(\Sigma))$ is dilatable.
\end{theorem}

\section{Dilatation of Fourier multipliers}
\subsection{Properties of Fourier multipliers}
This subsection provides some background on  Fourier multipliers
and is a preparation to the proof of Theorem (\textbf{B}). 
Let $G$ be a unimodular locally compact group. 
We use the notation introduced in the paragraph preceding
Theorem (\textbf{B}). In particular, we let
$\lambda : G\to B(L^2(G))$ be the left regular representation
defined by $[\lambda(s)f](t)=f(s^{-1}t)$. 
We also denote by
$\lambda : L^1(G)\to VN(G)\subset B(L^2(G))$ the mapping defined by 
$\lambda(f) = \int_G f(s)\lambda(s) ds$. (The use of
the same notation $\lambda$ for these two maps should
not create any confusion.) The above integral is 
defined in the strong sense. Indeed, for any 
$f_1\in L^1(G)$ and $f_2\in L^2(G)$,
$[\lambda(f_1)](f_2) = f_1 *f_2$, where
$$
(f_1 * f_2)(t) =\int_G f_1(s) f_2(s^{-1}t) ds,\qquad t\in G.
$$
We recall that for any $f_1,f_2\in L^2(G)$, the above
formula defines a function $f_1 * f_2\in C_0(G)$.

We note that if $T_u : VN(G)\to VN(G)$ is a Fourier multiplier 
with symbol $u : G\to \mathbb C$ which is a bounded continuous map, then 
$$
T_u(\lambda(f)) = \lambda(uf),\qquad f \in L^1(G).
$$

For any $f\in L^1(G)$, we define 
$f^\vee\in L^1(G)$ and  by 
$$
f^\vee(t)=f(t^{-1}),\quad t\in G\text{ and } f^*=\overline{\overset{\vee}{f}}
$$

\begin{lemme}\label{prop lambda lambda}
Let $f,h\in L^1(G)\cap L^2(G)$. We have:
\begin{enumerate}
\item $\lambda(f)^*\lambda(h)=\lambda(f^**h)$;
\item $\lambda(f)^*\lambda(h)\in L^1(VN(G))$ and $$
\omega_G(\lambda(f)^*\lambda(h))=\int_G\overline{f}(s)h(s)ds
=(f^*\ast h)(e),
$$
where $e$ denote the unit 
of $G$.
\end{enumerate}
\end{lemme}

\begin{proof}
See \cite[18.17 (5)]{Stratila}.
\end{proof}

It is plain that
the $*$-subalgebra $\lambda( L^1(G))$ is $w^*$-dense in $VN(G)$. 
Let $\mathcal{K}(G)\subset C_0(G)$ be the 
$*$-subalgebra of all continuous functions from $G$ into $\C$ with  
compact support. Then $\lambda(\mathcal{K}(G))$ is $w^*$-dense in $VN(G)$ 
(see \cite[3.12]{Folland} or \cite{Eymard}).

We recall that a function $u : G\to\C$ is called
positive definite if, for all $n\geq 1$, for all $t_1,\ldots,t_n\in G$ and for all $z_1,\ldots, z_n\in\C$, we have
$$
\sum_{i,j=1}^{n} u(t_i^{-1}t_j)\overline{z_i}z_j\,\geq 0.
$$

We recall  that 
a map $T: VN(G)\to VN(G)$ is called self-adjoint if for all $x,y\in VN(G)\cap L^1(VN(G))$, we have $\omega_G(T(x)y^*)=\omega_G(xT(y)^*)$. 

As in Section 3 on Schur multipliers, we relate the properties of a
Fourier multiplier $T_u$ to its symbol $u$.

\begin{theorem}\label{DCH}
Let $u:G\to \C$ be the symbol of a
Fourier multiplier $T_u: VN(G)\to VN(G)$. 
\begin{enumerate}
\item $T_u$ is completely positive if and only if $u$ is positive definite;
\item the map $T_u$ is unital if and only if $u(e)=1$;
\item the map $T_u$ is self-adjoint if and only if $u$ is real-valued.
\end{enumerate}
\end{theorem}

\begin{proof}
The assertion $(1)$ is given by \cite[Proposition 4.2]{Haagerup}, 
and $(2)$ is obvious. Let us prove $(3)$.
By density, 
$T_u$ is self-adjoint if and only if 
$$
\forall  f,h\in L^1(G)\cap L^2(G),\quad 
\omega_G ( (T_u(\lambda(f))^* \lambda(h))=\omega_G ( \lambda (f)^* T_u(\lambda (h))),
$$
if and only if
$$
\forall  f,h\in L^1(G)\cap L^2(G),\quad
\omega_G ( (\lambda(uf))^* \lambda(h))=\omega_G ( \lambda (f)^* \lambda (uh)).
$$
Using lemma \ref{prop lambda lambda}, this is equivalent to 
$$
\forall  f,h\in L^1(G)\cap L^2(G),\quad
\int_G\overline{u(s)}\overline{f(s)}h(s)ds=\int_G \overline{f(s)}u(s)h(s)ds.
$$
This is clearly equivalent to $u$ being real-valued.
\end{proof}

We recall that unlike for Schur Multiplier, there 
exist positive Fourier multipliers which are not completely positive, see \cite[Corollary 4.8]{Haagerup}. Similarly, a bounded Fourier multiplier is not necessarily completely bounded \cite{bozejko}.

\subsection{Proof of theorem (\textbf{B})}
The aim of this subsection is to prove theorem (\textbf{B}). We adapt the construction of \cite{A1} to our general setting. We assume that $T_u:VN(G)\to VN(G)$ is a completely positive, self-adjoint unital Fourier multiplier associated with a bounded continuous 
function $u: G\to\C$. Let  
$\Theta :\ell^1_\R(G)\times \ell^1_\R(G)\to \R$
be the bilinear symmetric map 
defined by
\begin{align*}
\Theta:(f,h)\mapsto \sum\limits_{G\times G} u(t^{-1}s)f(t)h(s).
\end{align*} 
This is well-defined, because $u$ is bounded, and $\Theta$ is positive by
Theorem \ref{DCH}, (1).

Let $K_\Theta\subset \ell^1_\R(G)$ be the kernel of the seminorm
$\Theta(f,f)^\frac12$ and let $H_u$ be the completion of
$\ell^1_\R(G)/K_\Theta$ for the norm induced by $\Theta(f,f)^\frac12$.
We denote by $\h$ the real Hilbert space $H_u\overset{2}{\otimes}\ell_\R^2(\Z)$, to which we associate
the tracial von Neumann algebra ($\Gamma_{-1}(\h),\tau)$.
Furthermore we let $H:=\mathcal{F}_1(\h)$ 
and we recall that by construction, $\G\subset B(H)$.

In the sequel, we let $\dot{f}\in H_u$ denote the class of any $f\in 
\ell^1_\R(G)$.

\begin{lemme}
For any $t\in G$, the map $\theta_t:\dot{f}\mapsto \dot{\overbrace{s\mapsto f(t^{-1}s)}}$ is an isometry from $H_u$ onto $H_u$.
\end{lemme}

\begin{proof}
This follows from the following equality. For all $f,h\in \ell_1(G)$, we have:
\begin{align*}
\sum_{G\times G} u(s'^{-1}s)f(ts)h(ts') =
\sum_{G\times G} u(s'^{-1}s)f(s)h(s') .
\end{align*}
\end{proof}

For all $t\in G$, the map $\theta_t\otimes Id_{\ell_\R^2(\Z)}$ 
extends to an onto isometry on $\h$. We denote this application by $\theta_t\overset{2}{\otimes} Id_{\ell_\R^2(\Z)}$. 
We obtain for all $t\in G$, the map $\Gamma_{-1}(\theta_t\overset{2}{\otimes} Id_{\ell_\R^2(\Z)})$ 
is a trace preserving $\ast$-automorphism.

We define a homomorphism 
\begin{align*}
\alpha : G\to Aut(\Gamma_{-1}(\mathcal H)),
\qquad \alpha(t) = 
\Gamma_{-1}(\theta_t\overset{2}{\otimes} Id_{\ell_\R^2(\Z)}).
\end{align*}

For any $t\in G$ we let $\delta_t\in \ell^1_\R(G)$ be 
defined by $\delta_t(t)=1$ and $\delta_t(s)=0$ if $s\not= t$.
According to (\ref{omega}), we have
\begin{equation}\label{commutation}
\alpha(t)\bigl(\omega(\dot{\delta_s}\otimes z)\bigr)
=\omega(\dot{\delta_{ts}}\otimes z),
\qquad t,s\in G,\ z\in l^2_\Z.
\end{equation}

We consider below 
continuity of $\alpha$ in the sense of \cite[Definition X.1.1]{Takesaki}  and 
\cite[Proposition X.1.2]{Takesaki}.

\begin{lemme} \label{continuité de alpha}
The map $\alpha:G\to Aut(\G)$ point-$w^*$-continuous, i.e for all $m\in \G$ and for all $\eta\in \G_*$, the map $t\mapsto \langle \alpha(t)(m),\eta\rangle_{\G,\G_*}$ is continuous.
\end{lemme}

\begin{proof}
Firstly we prove the continuity of the map
\begin{align*}
t\mapsto \langle\alpha(t)(\omega(\dot{\delta_{s_1}}\otimes h_1)\cdots\omega(\dot{\delta_{s_n}}\otimes h_n)),\omega(\dot{\delta_{v_1}}\otimes k_1)\cdots\omega(\dot{\delta_{v_n}}\otimes k_n)\rangle,
\end{align*}
for arbitrary $s_1,\ldots, s_n, v_1,\ldots, v_n\in G$ and
$h_1,\ldots,h_n,k_1,\ldots,k_n\in l^2_\Z$.

For all $t\in G$ we have, by (\ref{commutation}) and 
Lemma \ref{prop trace w w avec produit scalaire},
\begin{align*}
\langle\alpha(t)(\omega(\dot{\delta_{s_1}}\otimes h_1)&\cdots\omega(\dot{\delta_{s_n}}\otimes h_n)),
\omega(\dot{\delta_{v_1}}\otimes k_1)\cdots
\omega(\dot{\delta_{v_n}}\otimes k_n)\rangle\\
&=\langle\alpha(t)(\omega(\dot{\delta_{s_1}}\otimes h_1))\cdots\alpha(t)(\omega(\dot{\delta_{s_n}}\otimes h_n)),\omega(\dot{\delta_{v_1}}\otimes k_1)
\cdots\omega(\dot{\delta_{v_n}}\otimes k_n)\rangle\\
&=\langle(\omega(\dot{\delta_{ts_1}}\otimes h_1))\cdots(\omega(\dot{\delta_{ts_n}}\otimes h_n)),
\omega(\dot{\delta_{v_1}}\otimes k_1)\cdots\omega(\dot{\delta_{v_n}}\otimes k_n)\rangle\\
&=\tau\bigl(\omega(\dot{\delta_{ts_1}}\otimes
h_1)\cdots\omega(\dot{\delta_{ts_n}}\otimes h_n)
\omega(\dot{\delta_{v_1}}\otimes k_1)
\cdots\omega(\dot{\delta_{v_n}}\otimes k_n)\bigr)\\
&=\sum_{\nu\in \mathcal{P}_2(2n)}(-1)^{c(\nu)}\prod_{i,j\in \nu}\langle \dot{\delta_{t_i}},\dot{\delta_{t_j}}\rangle_{H_u}\langle l_i,l_j\rangle_{l^2_\Z},
\end{align*}
where $t_i=\left\lbrace\begin{array}{ccc}
ts_i&\text{if}&i\leq n\\
v_{i-n}&\text{if}& i>n
\end{array}\right. $ and $l_i=\left\lbrace\begin{array}{ccc}
h_i&\text{if}&i\leq n\\
k_{i-n}&\text{if}& i>n
\end{array}\right. $.\\

We have $\langle\dot {\delta_{ts_i}},
\dot {\delta_{v_j}}\rangle_{H_u}=u((ts_i)^{-1}v_j)$, and 
$u$ is continuous. Hence we obtain the expected continuity.

Next we note that by the density of $\text{span}\lbrace \omega(.)\omega(.)\cdots\omega(.)\rbrace$ in $L^2(\G)$, 
this implies continuity of the map 
$t\mapsto \langle \alpha(t)a,b\rangle$  for all $a,b\in L^2(\G)$. 
Finally let $m\in \G$ and let $\eta\in \G_*=L^1(\G)$.
We may write $\eta=ab$, 
with $a,b\in L^2(\G)$. Then
\begin{align*}
\langle \alpha(t)m,\eta\rangle
=\langle \alpha(t)m,ab\rangle=\langle \alpha (t) ma,b\rangle.
\end{align*}
The continuity of $t\mapsto \langle \alpha(t)m,\eta\rangle$ follows.
\end{proof}

We give some background about the crossed product 
$\G\rtimes_\alpha  G$ associated with $\alpha$.
Define, for all $m\in \G$,
\[\pi_\alpha(m):\xi\in L^2(G,H)
\mapsto \alpha(.)^{-1}(m)(\xi)(.)\in L^2(G,H). \]
Then $\pi_\alpha$  is a 1-1 $*$-homomorphism
from $\G$ into $B(L^2(G,H))$. We note in
passing that for every $m\in \G$ the map 
$s\mapsto\alpha(s^{-1})(m)$ is a element of 
the space $L^\infty_\sigma(G,B(H))$. 
By definition,
$$
\G\rtimes_\alpha  G:=VN\bigl\lbrace \lbrace \pi_\alpha(m);\,m\in \G\rbrace,\lbrace x\overline{\otimes} I_H:\,x\in VN(G)\rbrace\rbrace\subset  
B(L^2(G,H)).
$$
We define 
\begin{equation}\label{J}
J :VN(G)\longrightarrow \G\rtimes_\alpha  G,\qquad 
J(x)=x\overline{\otimes}I_H.
\end{equation}
This one-to-one 
$\ast$-homomorphism is the one to be used in 
Theorem (\textbf{B}).

The following is a classical fact.

\begin{lemme}\label{prop commutationJ et pi}
For any $x\in \G$ and $t\in G$,
\begin{align*}
J(\lambda(t))\pi_\alpha(x)=\pi_\alpha(\alpha(t)(x))J(\lambda(t)).
\end{align*}
\end{lemme}

This following will be applied a few times in this paper.

\begin{theorem}\label{lemme egalite de poids}
Let $\varphi$ and $\psi$ two n.s.f. traces on a von Neumann algebra $\mathcal{M}$. 
Assume that there exists a $w^*$-dense $*$-subalgebra $B$ of $\mathcal{M}$
such that $B\subset  L^2(\mathcal{M},\varphi)$ and for all $y\in B$,
\[\psi(y^*y)=\varphi(y^*y).\]
Then $\varphi=\psi$.
\end{theorem}
\begin{proof}
It is a special case of the Pedersen-Takesaki theorem, see \cite[§6.2]{Stratila}.
\end{proof}

We let $\tau_M$ denote 
the dual weight on $M= \G\rtimes_\alpha  G$, 
for which we refer
to \cite[p. 61, p. 248-249]{Takesaki}.
The action $\alpha$ considered in the present paper
is trace preserving
hence the n.s.f. weight $\tau_M$ is actually 
a trace. Indeed it follows from in \cite[Theorem X.1.17]{Takesaki} 
that the 
modular group of $\tau_M$ acts trivially on $\pi(M)$ and 
on $J(VN(G))$.

We now recall a construction which is a slight variant of the 
one given in \cite[Section X]{Takesaki}. 
Let $\KG$ be the vector space of all compactly supported 
$\sigma$-strongly*-continuous  
$\G$-valued functions on $G$. 
For all $F,F'\in\KG $, we define $F*F'\in\KG$ and $F^\sharp
\in\KG$ by
\begin{align*}
F*F'(t)=\int_GF(s^{-1})\alpha^{-1}(s)(F'\left(st\right))ds
\quad\hbox{and}\quad
F^\sharp(t)=\alpha(t)(F(t^{-1})^*).
\end{align*}
Then for any $F\in\KG$, we define 
$T_F^\alpha\in \G$ by
\begin{align*}
\forall \eta\in \G_*,\text{ } \langle T^\alpha_F,\eta\rangle=
\int_G\langle\pi_\alpha(F(t))J(\lambda(t)),\eta\rangle dt.
\end{align*} 
We let $\mathcal{B}\subset \G$ denote the $*$-subalgebra of all 
$T_F^\alpha$, for $F\in \KG$.

Following \cite[Lemma X.1.8]{Takesaki}, we have the following two results.

\begin{lemme}\label{prop formule prod et sharp}
For all $F,F_1,F_2\in\KG$,
\begin{align*}
(T_F^\alpha)^*=T_{F^\sharp}^\alpha \text{ and } T_{F_1}^\alpha T_{F_2}^\alpha=T_{F_1*F_2}^\alpha
\end{align*}
\end{lemme}

\begin{lemme}\label{prop B dense}
\begin{enumerate}
\item[]
\item The $*$-subalgebra $\mathcal{B}$ is $w^*$- dense in $\prodalpha$;
\item For all $F\in \KG$, we have $T^\alpha_{F^\sharp*F}\geq 0$. 
\item Let $e$ be the unit of $G$. If $F\in \KG$ is
such that $T_F^\alpha\geq 0$, then:
\begin{align*}
\tau_M(T_F^\alpha)=\tau(F(e)).
\end{align*}
\end{enumerate}
\end{lemme}

We now prove a property
necessary to apply Definition \ref{rmq superdilatation}.

\begin{lemme}\label{TrPres}
The map $J$ defined by (\ref{J}) is trace preserving, that is,
\begin{align*}
\tau_M\circ J=\omega_G.
\end{align*}
\end{lemme}

\begin{proof}
Firstly we remark that for all $f\in \mathcal{K}(G)$, $T^\alpha_{f^**f\otimes 1}=J(\lambda(f^**f))$. Using Lemma \ref{prop lambda lambda}, it follows that
\begin{align*}
\tau_M(J(\lambda(f)^*\lambda(f)))=\tau_M(J(\lambda(f^**f)))=\tau_M(T^\alpha_{f^**f\otimes 1})).
\end{align*}
According to Lemma \ref{prop B dense}, we have
\begin{align*}
\tau_M(T^\alpha_{f^**f\otimes 1})=\tau(f^**f\otimes 1(e))=\tau(1)f^**f(e)=\omega_G(\lambda(f)^*\lambda(f)).
\end{align*}
Consequently,
\begin{align*}
\tau_M(J(\lambda(f)^*\lambda(f)))=\omega_G(\lambda(f)^*\lambda(f)).
\end{align*}
The result follows by applying Theorem \ref{lemme egalite de poids} 
to $\omega_G$ and the n.s.f. trace $\tau_M\circ J$.
\end{proof}

We now establish a link between the trace $\tau_M$
on $M=\prodalpha$, the Plancherel trace $\omega_G$
and the trace $\tau$ on $\G$.

\begin{prop}\label{prop egalite tauM=omegatau}
For any $x\in \G$ and $m\in VN(G)\cap L^1(VN(G))$,
$\pi_\alpha(x)J(m)$ belongs to $M\cap L^1(M)$ and we have
\begin{align*}
\tau_M(\pi_\alpha(x)J(m))=\omega_G(m)\tau(x).
\end{align*}
\end{prop}

We first establish 
the above result in a special case.

\begin{lemme}\label{special}
Let $x\in \G$ and $f\in \mathcal{K}(G)$. We have:
\begin{align*}
\tau_M(\pi_\alpha(x^*x)J(\lambda(f)^*\lambda(f))=\tau(x^*x)\omega_G(\lambda(f)^*\lambda(f)).
\end{align*}
\end{lemme}

\begin{proof}
We have by the tracial property of $\tau_M$:
\begin{align*}
\tau_M(\pi_\alpha(x^*x)J(\lambda(f)\lambda(f)^*))
&=\tau_M(\pi_\alpha(x)^*\pi_\alpha(x)J(\lambda(f))J(\lambda(f))^*)\\
&\begin{array}{ccc}
=\tau_M(&\underbrace{\pi_\alpha(x)J(\lambda(f))(\pi_\alpha(x)J(\lambda(f)))^*}&)\text{ }\\
&\geq 0.
\end{array}
\end{align*}
We observe that $T^\alpha_{f\otimes x}= \pi_\alpha(x)J(\lambda(f))$ and $(f\otimes x)*(f\otimes x)^\sharp=f*f^*(.) x\alpha(.)(x^*)$. Then we have, 
using Lemmas \ref{prop lambda lambda}, \ref{prop formule prod et sharp} and
\ref{prop B dense},
\begin{align*}
\tau_M(\pi_\alpha(x^*x)J(\lambda(f)\lambda(f)^*))&=\tau_M(T^\alpha_{f\otimes x}(T^\alpha_{f\otimes x})^*)\\
&=\tau_M(T^\alpha_{(f\otimes x)*(f\otimes x)^\sharp})\\
&=\tau_M(T^\alpha_{f*f^*(.) x\alpha(.)(x^*)})\\
&=\tau(f*f^*(e)\otimes x\alpha(e)(x^*))\\
&=f*f^*(e)\tau(xx^*)\\
&=\omega_G(\lambda(f)\lambda(f)^*)\tau(x^*x).
\end{align*}
We obtain the result by changing $f$ into $f^*$.
\end{proof}


\begin{proof}[Proof of Proposition \ref{prop egalite tauM=omegatau}]
For any $m\in VN(G)\cap L^1(VN(G))$, $J(m)\in M\cap L^1(M)$, by Lemma 
\ref{TrPres}. Hence
$\pi_\alpha(x)J(m)$ belongs to $M\cap L^1(M)$ for 
any $x\in \G$. 

We fix $x\in \G$ and we may assume that $x\neq 0$. 
We define a n.s.f. trace $\Phi_x$ on $VN(G)$ by
\begin{align*}
\Phi_x(m) =\tau_M\left(\dfrac{\pi_\alpha(x^*x)}{\tau(x^*x)}J(m)\right),
\qquad m\in VN(G)_+.
\end{align*}
By Lemma \ref{special}, $\Phi_x$ coincides with
$\omega_G$ on 
$\lbrace \lambda(f)^*\lambda(f):$ $f\in \mathcal{K}(G)\rbrace$.
Appying Theorem \ref{lemme egalite de poids}, we deduce that 
$\Phi_x=\omega_G$. Thus we have
\begin{align*}
\tau_M(\pi_\alpha(x^*x)J(m))=\tau(x^*x)\omega_G(m),
\qquad x\in \G,\ m\in VN(G)_+.
\end{align*}
By polarization, the result follows at once.
\end{proof}

As in Section 4 on Schur multipliers, $J$ induces $J_1:L^1(VN(G))\to L^1(\prodalpha)$ and we denote by $\mathbb{E}=J_1^*:\prodalpha\to VN(G)$ the 
conditional expectation associated with $J$.

\begin{lemme}\label{prop esperance simplification}
For any $x\in \G$ and $t\in G$, we have:
\begin{align*}
\mathbb{E}(\pi_\alpha(x)J(\lambda(t)))=\tau(x)\lambda(t).
\end{align*}
\end{lemme}

\begin{proof}
Let $\eta\in VN(G)\cap L^1(VN(G))$. Using 
Proposition \ref{prop egalite tauM=omegatau}, we have:
\begin{align*}
\langle \mathbb{E}(\pi_\alpha(x)J(\lambda(t))),
\eta\rangle&=\langle\pi_\alpha(x)J(\lambda(t)),J_1(\eta)\rangle\\
&=\langle \pi_\alpha(x)\lambda(t)\overline{\otimes}1,\eta\overline{\otimes} 1\rangle\\
&=\tau_M(\pi_\alpha(x)(\lambda(t)\eta\overline{\otimes}1))\\
&=\tau_M(\pi_\alpha(x)J(\lambda(t)\eta))\\
&=\omega_G(\lambda(t)\eta)\tau(x).
\end{align*} 
Since $VN(G)\cap L^1(VN(G))$ is dense in $L^1(VN(G))$, this yields the result.
\end{proof}

Let $\rho\in B(l^2_\Z)$ be the shift operator,
$\rho\bigl((a_n)_n\bigr)=\bigl((a_{n+1})_n\bigr).$
This is a unitary hence $\rho':=\Gamma_{-1}(Id\overset{2}{\otimes}\rho)$
is a normal completely positive,  trace preserving automorphism.
Further we have
\begin{align}\label{equalite rho omega}
\rho'(\omega(x))=\omega(Id\overset{2}{\otimes}\rho (x)),
\qquad x\in \G,
\end{align}
by (\ref{omega}).
It follows that for any $t\in G$, we have:
\[ \alpha(t)\rho'=\rho'\alpha(t).\]

\begin{theorem}\label{th application sunder}
There exists a (necessarily unique) automorphism $\tilde{\rho}:\G\rtimes_\alpha  G\to \G\rtimes_\alpha  G$ such that:
\begin{enumerate}
\item $\tilde{\rho}$ is trace preserving;
\item for all $x\in \G$, $\tilde{\rho}(\pi_\alpha(x))=\pi_\alpha(\rho'(x))$;
\item for all $t\in G$, $\tilde{\rho}(J(\lambda(t)))=J(\lambda(t))$.
\end{enumerate}
\end{theorem}

\begin{proof}
We apply \cite[Theorem 4.4.4 p. 149]{Sunder}. 
It provides all the expected properties except the 
fact that $\tilde{\rho}$ is trace preserving. Let us check this.

Let $F\in \KG$. We observe that for all $t\in G$:
\begin{align*}
\rho'(F^\sharp*F(t))
&=\rho'\left(\int_G\alpha^{-1}(s)(F(s)^*)\alpha^{-1}(s)(F(ts))ds\right)\\
& =\int_G\alpha^{-1}(s)((\rho'(F(s)))^*)\alpha^{-1}(s)(\rho'(F(ts))ds.
\end{align*}
The map $\rho'$ is a automorphism and $\alpha \text{ and }\rho'$ commute, 
hence we have:
\begin{align*}
\rho'(F^\sharp*F(t))=(\rho'\circ F)^\sharp*(\rho'\circ F)(t)
\end{align*}
We have for all $\eta\in (\prodalpha)_*$:
\begin{align*}
\langle\tilde{\rho}((T_F^\alpha)^*T_F^\alpha),\eta\rangle&=\int_G \langle\tilde{\rho}(\pi_\alpha(F^\sharp*F)(t))\tilde{\rho}(J(\lambda(t)),\eta\rangle dt\\
&=\int_G\langle\pi_\alpha(\rho'(F^\sharp*F(t))J(\lambda(t)),\eta\rangle dt\\
&=\int_G\left\langle\pi_\alpha((\rho'\circ F)^\sharp*(\rho'\circ F)(t))J(\lambda(t)),\eta\right\rangle dt\\
&=\langle(T_{\rho'\circ F}^\alpha)^*T_{\rho'\circ F}^\alpha,\eta\rangle.
\end{align*}
Hence $\tilde{\rho}((T^\alpha_F)^*T^\alpha_F)=(T^\alpha_{\rho'\circ F})^*T^\alpha_{\rho'\circ F}$. Computing the trace, we have:
\begin{align*}
\tau_M(\tilde{\rho}((T_F^\alpha)^*T_F^\alpha)&)=\tau_M((T_{\rho'\circ F}^\alpha)^*T_{\rho'\circ F}^\alpha)\\
&=\tau((\rho'\circ F)^*(\rho'\circ F)(e))\\
&=\tau(\rho'(F^\sharp*F(e)))\\
&=\tau(F^\sharp*F)(e))\\
&=\tau_M((T^\alpha_F)^*T^\alpha_F).
\end{align*}
Using the $w^*$-density of $\mathcal{B}$ and 
Theorem \ref{lemme egalite de poids}, we obtain the equality $\tau_M\circ \tilde{\rho}=\tau_M$. 
\end{proof}

In the sequel we let $(\varepsilon_k)_{k\in\Z}$ 
denote the standard  basis of $l^2_\Z$.

\begin{lemme}
Let $d=\pi_\alpha(\omega(\dot{\delta_e}\otimes \varepsilon_0))
\in \KG$,
where $e$ is the unit of $G$. Then $d^*=d$ and $d^2=1$.
\end{lemme}

\begin{proof}
It is clear that $d^*=d$. Next observe that since $T_u$ is unital, we have $u(e)=1$ hence
\begin{align*}
\|\dot{\delta_e}\otimes \varepsilon_0\|_\h=\|\dot{\delta_e}\|_{H_u}=u(e)^\frac{1}{2}=1.
\end{align*}
Consequently, $d^2=1$ by (\ref{eq équilité norme}).
\end{proof}

We define
$$
U : \G\rtimes_\alpha  G\longrightarrow
\G\rtimes_\alpha  G,\qquad
U(x) =d^*\tilde{\rho}(x)d.
$$
It follows from above that this is trace preserving
automorphism.

\begin{lemme}\label{lemme ukj simplification}
For all $k\in \N_0$, for all $t\in G$, we have:
\begin{align*}
& U^k(J(\lambda(t)))=\\
&\pi_\alpha(\omega(\dot{\delta_e}\otimes \varepsilon_0)\omega(\dot{\delta_e}\otimes \varepsilon_1)\cdots\omega(\dot{\delta_e}\otimes \varepsilon_{k-1})\omega(\dot{\delta_t}\otimes \varepsilon_{k-1})\cdots\omega(\dot{\delta_t}\otimes \varepsilon_1)\omega(\dot{\delta_t}\otimes \varepsilon_0))J(\lambda(t)).
\end{align*}
\end{lemme}

\begin{proof}
We prove the lemma by induction. The case $k=0$ is easy. Let $k\in N_0$
and assume that the result true for $k$. We remark that for all $t\in G$ and $x\in \G$:
\begin{align*}
U(\pi_\alpha(x)J(\lambda(t)))=\pi_\alpha(\omega(\dot{\delta_e}\otimes \varepsilon_0))\pi_\alpha(\rho'(x))J(\lambda(t))\pi_\alpha(\omega((\dot{\delta_e}\otimes \varepsilon_0))).
\end{align*}
Hence for any $t\in G$ we have:
\begin{align*}
&U^{k+1}(J(\lambda(t)))\\
&=U(\pi_\alpha(\omega(\dot{\delta_e}\otimes \varepsilon_0)\omega(\dot{\delta_e}\otimes \varepsilon_1)\cdots\omega(\dot{\delta_e}\otimes \varepsilon_{k-1})\omega(\dot{\delta_t}\otimes \varepsilon_{k-1})\cdots\omega(\dot{\delta_t}\otimes \varepsilon_1)\omega(\dot{\delta_t}\otimes \varepsilon_0))J(\lambda(t))\\
&=\pi_\alpha(\omega(\dot{\delta_e}\otimes \varepsilon_0))\pi_\alpha(\rho'(\omega(\dot{\delta_e}\otimes \varepsilon_0)\cdots\omega(\dot{\delta_t}\otimes \varepsilon_0)))J(\lambda(t))\pi_\alpha(\omega(\dot{\delta_e}\otimes \varepsilon_0)).
\end{align*}
By \eqref{equalite rho omega}, this is equal to
\begin{align*}
\pi_\alpha(\omega(\dot{\delta_e}\otimes \varepsilon_0)\omega(\dot{\delta_e}\otimes \varepsilon_1)\cdots\omega(\dot{\delta_e}\otimes \varepsilon_{k})\omega(\dot{\delta_t}\otimes \varepsilon_{k})\cdots\omega(\dot{\delta_t}\otimes \varepsilon_1)))J(\lambda(t))\pi_\alpha(\omega(\dot{\delta_e}\otimes \varepsilon_0)).
\end{align*}
By lemma \ref{prop commutationJ et pi}, the latter is equal to
\begin{align*}
&\pi_\alpha(\omega(\dot{\delta_e}\otimes \varepsilon_0)\omega(\dot{\delta_e}\otimes \varepsilon_1)\cdots\omega(\dot{\delta_e}\otimes \varepsilon_{k})\omega(\dot{\delta_t}\otimes \varepsilon_{k})\cdots\omega(\dot{\delta_t}\otimes \varepsilon_1)))\pi_\alpha(\omega(\dot{\delta_t}\otimes \varepsilon_0))J(\lambda(t))\\
&=\pi_\alpha(\omega(\dot{\delta_e}\otimes \varepsilon_0)\omega(\dot{\delta_e}\otimes \varepsilon_1)\cdots\omega(\dot{\delta_e}\otimes \varepsilon_{k})\omega(\dot{\delta_t}\otimes \varepsilon_{k})\cdots\omega(\dot{\delta_t}\otimes \varepsilon_1)\omega(\dot{\delta_t}\otimes \varepsilon_0))J(\lambda(t)).
\end{align*}
This proves the lemma.
\end{proof}

We conclude the proof of Theorem ({\bf B}) by showing:
\begin{align*}
\mathbb{E}U^kJ=T_u^k, \qquad k\geq 0.
\end{align*} 
It suffices to prove that for all $t\in G$, $\mathbb{E}U^kJ(\lambda(t))=T^k(\lambda(t))$. 
To check this, fix $t\in G$ and 
$k\in N_0$ and for all $i\in [|1;2k|]$,
consider $f_i=\left\lbrace\begin{array}{ccc}
\dot{\delta_e}\otimes \varepsilon_{i-1}&\text{ If }& 1\leq i\leq k\\
\dot{\delta_t}\otimes \varepsilon_{2k-i}&\text{ If }& k+1\leq i\leq 2k
\end{array}\right.$.

Then $\langle f_i,f_{2k-i+1}\rangle_{-1}=u(t)$ and
for all $1\leq i<j\leq 2k$ such that $j\neq 2k-i+1$, $\langle f_i,f_j\rangle_\h=0$.

We have according to  Lemma \ref{lemme ukj simplification} :
\begin{align*}
\mathbb{E}U^kJ(\lambda(t))&=\mathbb{E}(\pi_\alpha(\omega(\dot{\delta_e}\otimes \varepsilon_0)\omega(\dot{\delta_e}\otimes \varepsilon_1)\cdots\omega(\dot{\delta_t}\otimes \varepsilon_1)\omega(\dot{\delta_t}\otimes \varepsilon_0))J(\lambda(t))).
\end{align*}
Using Lemma \ref{prop esperance simplification}, this is equal to 
\begin{align*}
\ &\tau(\omega(\dot{\delta_e}\otimes \varepsilon_0)\omega(\dot{\delta_e}\otimes \varepsilon_1)\cdots\omega(\dot{\delta_e}\otimes \varepsilon_{k-1})\omega(\dot{\delta_t}\otimes \varepsilon_{k-1})\cdots\omega(\dot{\delta_t}\otimes \varepsilon_1)\omega(\dot{\delta_t}\otimes \varepsilon_0))\lambda(t)\\
&=\tau(\omega(f_1)\cdots\omega(f_{2k}))\lambda(t)=\langle f_1,f_{2k}\rangle_\h\cdots\langle f_k,f_{k+1}\rangle_\h \lambda(t).
\end{align*}
We derive, using lemma \ref{prop trace w w avec produit scalaire}, that
\begin{align*}
\mathbb{E}U^kJ(\lambda(t))=u(t)^k\lambda(t)=T_u^k(\lambda(t)),
\end{align*}
which concludes the proof.

By Lemma \ref{lem absdila implique dila p}, we have the following
explicit consequence of Theorem (\textbf{B}).

\begin{theorem}
Let $G$ be a unimodular locally compact group. 
Let $T: VN(G)\to VN(G)$ is a self-adjoint, unital, completely positive Fourier multiplier.
Then for all $1\leq p<\infty$, the map $T_p:L^p(VN(G))\to L^p(VN(G))$ is dilatable.
\end{theorem}

In a more general approach, with a locally compact group (so without the unimodular hypothesis), it is possible to obtain a similar result, but in the framework of a von Neumann algebra equipped with a weight. It will be written in \cite{thesecd}.

\section{Multi-variable dilatations}
This last section is devoted to  Theorems (\textbf{C}) and (\textbf{D}). The approach is  similar to the one in \cite{SKR}, where the discrete case was considered. Since many arguments are similar to the ones in Sections
4 and 5, we will not write all the details.

\subsection{Case of Schur multipliers}
The goal of this
section is to prove Theorem (\textbf{C}); then we
will state a consequence (Theorem \ref{Cons1}).
To make the presentation simpler,
we only consider the case $n=2$. The arguments for 
the general case are similar. So
we consider $\varphi_1,\varphi_2\in L^\infty(\Sigma^2)$ and we assume that $M_{\varphi_1}, M_{\varphi_2}$ are self-adjoint unital positive Schur multipliers on $B(L^2(\Sigma))$. 
We let $\mathbb{H}_1$, respectively $\mathbb{H}_2$, denote 
the real Hilbert space associated to
$\varphi_1$, respectively $\varphi_2$, as
at the beginning of Section 4. Then we consider
the tracial von Neumann algebras $(N_1,\tau_1):=(\Gamma_{-1}(\mathbb{H}_1),\tau_1)$ and $(N_2,\tau_2):=(\Gamma_{-1}(\mathbb{H}_2),\tau_2)$, and we
define 
$$
(N,\tau):=(N_1\overline{\otimes}N_2,
\tau_1\overline{\otimes}\tau_2). 
$$
Next we let 
$T_1\in B(L^1(\Sigma),N_1)$ and $T_2\in B(L^1(\Sigma),N_2)$, be
the unique linear maps such that 
\begin{align*}
T_1(h)= \omega_1(\dot{h})
\quad\text{and}\quad
T_2(h) =\omega_2(\dot{h}),\qquad h\in L^1_{\R}(\Sigma).
\end{align*}
For $i=1,2$, we let $F_i\in L^\infty_\sigma(\Sigma,N_i)$
be associated to $T_i$.
In addition we define $T_1^\otimes$ and $T_2^\otimes$ in $B(L^1(\Sigma),N)$ by
\begin{align*}
T_1^\otimes(h)= T_1(h)\otimes 1_{N_2} 
\quad\text{and}\quad
T_2^\otimes(h)= 1_{N_1}\otimes T_2(h),\qquad h\in L^1(\Sigma).
\end{align*}
As in Section 4, we consider the infinite tensor product
$N^\infty$, and we now let $T_1^\infty$
and $T_2^\infty$ in $B(L^1(\Sigma),N^\infty)$
be the natural extensions of $T_1^\otimes$ and $T_2^\otimes$,
respectively.
Finally we let $F_1^\infty$ and $F_2^\infty$ in 
$L^\infty_\sigma(\Sigma,N^\infty)$ be corresponding to
$T_1^\infty$ and $T_2^\infty$, respectively, in the identification $B(L^1(\Sigma),N^\infty)\simeq
L^\infty_\sigma(\Sigma,N^\infty)$. Likewise
we let $d_1$ and $d_2$ in 
$L^\infty(\Sigma)\overline{\otimes}N^\infty$ be corresponding to
$T_1^\infty$ and $T_2^\infty$, respectively.
It follows from Section 4 that $d_1,d_2$ are self-adjoint
symmetries. Moreover $d_1$ and $d_2$ commute, by construction.

We consider the shifts $\mathcal{S}_1, \mathcal{S}_2$ on $N^\infty$ such that 
\begin{align*}
\mathcal{S}_1(\cdots\otimes(x_0\otimes y_0)\otimes (x_1\otimes y_1)\otimes\cdots):=\cdots \otimes (x_{-1}\otimes y_0)\otimes (x_0\otimes y_1)\otimes\cdots
\end{align*}
and
\begin{align*}
\mathcal{S}_2(\cdots\otimes(x_0\otimes y_0)\otimes (x_1\otimes y_1)\otimes\cdots):=\cdots \otimes (x_0\otimes y_{-1})\otimes (x_1\otimes y_0)\otimes\cdots
\end{align*}

We set $\mathcal{M}:=B(L^2(\Sigma))\overline{\otimes}N^\infty$ and we let 
$J\colon B(L^2(\Sigma))\to \mathcal{M}$ be defined by 
$J(x)=x\otimes 1_{N^\infty}$. 
This is a trace preserving one-to-one $\ast$-homomorphism.
We let $\E : \mathcal{M}\to B(L^2(\Sigma))$ denote the conditional 
expectation associated with $J$.

For $i=1,2$, we define a trace preserving $*$-automorphism 
$U_i : {\mathcal M}\to {\mathcal M}$ by
\begin{align*}
U_i(y) =d_i((Id\otimes\mathcal{S}_i(y))d_i.
\end{align*}
Let us now prove that $U_1$ and $U_2$ commute. We already have $d_1d_2=d_2d_1$. 
We remark that $d_1\in B(L^2(\Sigma))\overline{\otimes}N_1^\infty\overline{\otimes}1$ and $d_2\in B(L^2(\Sigma))\overline{\otimes}1\overline{\otimes}N_2^\infty$.
On the other hand, We can see $Id\otimes \mathcal{S}_1$ (respectively $Id\otimes \mathcal{S}_2$) as being the same as 
$Id\otimes \mathcal{S}\overline{\otimes}Id$
(respectively $Id\otimes Id\overline{\otimes}\mathcal{S}$). Further 
we have $Id\otimes \mathcal{S}_1(d_2)=d_2$ and $Id\otimes \mathcal{S}_2(d_1)=d_1$. 
Hence for any $x\in\mathcal M$, we have
\begin{align*}
U_1U_2(x)&=d_1(Id\otimes\mathcal{S}_1(d_2(Id\otimes \mathcal{S}_2(x))d_2)d_1\\
&=d_1(Id\otimes\mathcal{S}_1(d_2)Id\otimes\mathcal{S}_1(Id\otimes \mathcal{S}_2(x))Id\otimes\mathcal{S}_1(d_2)d_1\\
&=d_1d_2Id\otimes\mathcal{S}_1(Id\otimes \mathcal{S}_2(x))d_2d_1\\
&=d_2d_1Id\otimes\mathcal{S}_2(Id\otimes \mathcal{S}_1(x))d_1d_2\\
&\downarrow \text{ same computation}\\
&=U_2U_1(x).
\end{align*}
This proves the commutation property.

We now apply Lemma \ref{prop existence gamma}
and we let
$$
\Gamma : L^2_\sigma(\Sigma^2, N^\infty)\longrightarrow B(L^2(\Sigma))\overline{\otimes}N^\infty
$$
be the resulting $w^*$-continuous contraction.
Applying the argument in the proof of
Lemma \ref{lemme d et gamma}, we have:

\begin{lemme}

\ 

\begin{enumerate}
\item
For all $\theta\in L^2(\Sigma^2)$ and for all $y=\cdots\otimes y_{-1}\otimes(1_{N_1}\otimes y_0^2)\otimes y_1\cdots\in N^\infty$,
\begin{align*}
d_1\Gamma(\theta\otimes y)d_1=\Gamma\left((s,t)\mapsto \theta(s,t)[\cdots\otimes y_{-1}\otimes( \tilde{F_1\times F_1}(s,t)\otimes y_0^2)\otimes y_1\otimes\cdots]\right)
\end{align*}
\item 
For all $\theta\in L^2(\Sigma)$ and for all $y=\cdots\otimes y_{-1}\otimes(y_0^1\otimes 1_{N_2})\otimes y_1\cdots\in N^\infty$:
\begin{align*}
d_2\Gamma(\theta\otimes y)d_2=\Gamma\left((s,t)\mapsto \theta(s,t)[\cdots\otimes y_{-1}\otimes( y_0^1\otimes\tilde{F_2\times F_2}(s,t))\otimes y_1\otimes\cdots]\right)
\end{align*}
\end{enumerate}
\end{lemme}

This is all that we need. Indeed,
using this lemma and computations similar to the ones 
in the last part of Section 4,
we obtain that for all integers 
$k,l\geq 0$ and for all $f,h\in L^2(\Sigma)$, we have
$$
{\mathbb E}U_1^k U_2^l J(S_{f\otimes h}) = 
M_{\varphi_1}^kM_{\varphi_2}^l
(S_{f\otimes h}).
$$
The full statement of Theorem (\textbf{C})
follows at once.

We finally mention that 
applying Lemma \ref{lem superdilatation}, we obtain the
following consequence.

\begin{theorem}\label{Cons1}
Let $\varphi_1,\dots,\varphi_n\in L^\infty(\Sigma^2)$ and assume each $M_{\varphi_i}$ is a self-adjoint, unital, positive Schur multiplier on $B(L^2(\Sigma))$. Then, 
for all $1\leq p<\infty$, there exist a tracial von Neumann algebra $(M,\tau)$, a
commuting $n$-tuple $(U_1,\dots,U_n)$ of surjective isometries  
on $L^p(M)$, and two contractions $J: S^p(L^2(\Sigma))\to L^p(M)$ and $Q:L^p(M)\to S^p(L^2(\Sigma))$ such that for all $k_1,\ldots, k_n$
in $\N_0$,
\begin{align*}
M_{\varphi_1}^{k_1}\cdots M_{\varphi_n}^{k_n}=
QU_1^{k_1}\cdots U_n^{k_n}J
\qquad \hbox{on}\ S^p(L^2(\Sigma)).
\end{align*}
\end{theorem}

\subsection{Case of Fourier multipliers}
We now aim at proving Theorem (\textbf{D}). 
As in the case of Theorem (\textbf{C})
we only treat the case $n=2$. 
So we consider two 
self-adjoint, unital, completely positive Fourier multipliers
$T_{u_1},T_{u_2}$ on $VN(G)$. As we did at the beginning
of Section 5.2, we create real Hilbert spaces
$H_{u_1}$, $H_{u_2}$, $\mathcal{H}_1=H_{u_1}\overset{2}{\otimes}\ell_\R^2(\Z)$, $\mathcal{H}_2=H_{u_2}\overset{2}{\otimes}\ell_\R^2(\Z)$,
as well as the tracial von Neumann algebras
$(\Gamma_{-1}(\mathcal{H}_1),\tau_1) $ and $(\Gamma_{-1}(\mathcal{H}_2),\tau_2)$. Futher we
consider the resulting tracial  von Neumann algebra  
$$
\Gamma_{-1}(\mathcal{H}_1)
\overline{\otimes}\Gamma_{-1}(\mathcal{H}_2)\subset B(H).
$$
Next for any $t\in G$,
we consider, for $i=1,2$, the onto isometry
$\theta_t^{i}$ on $H_{u_i}$ given by
$$
\theta^i_t(\dot{f}) :  \dot{\overbrace{s\mapsto f(ts)}},
$$
and the trace preserving $*$-automorphism  
$$
\Gamma_{-1}(\theta^i_g\overset{2}{\otimes} Id_{\ell_\R^2(\Z)})
: \Gamma_{-1}(\mathcal{H}_i)\longrightarrow
\Gamma_{-1}(\mathcal{H}_i).
$$
Then we define 
$$
\alpha : G\longrightarrow 
Aut(\Gamma_{-1}(\mathcal{H}_1)
\overline{\otimes}\Gamma_{-1}(\mathcal{H}_2))
$$
by
\begin{align*}
\alpha(t) = \Gamma_{-1}(\theta^1_t\overset{2}{\otimes} Id_{\ell_\R^2(\Z)})\overline{\otimes}
\Gamma_{-1}(\theta^2_t\overset{2}{\otimes} Id_{\ell_\R^2(\Z)}),
\qquad t\in G.
\end{align*}
This is a  homomorphism, and this homomorphism
is point-$w^*$-continuous (in the sense of lemma \ref{continuité de alpha}),
by an argument similar to the one used in Section 5. This
gives rise 
to the
crossed product 
$$
M=\Gamma_{-1}(\mathcal{H}_1)\overline{\otimes}\Gamma_{-1}(\mathcal{H}_2)\rtimes_\alpha  G.
$$
As in Section 5, we see that the dual weight 
on this crossed product is a trace. 
We let 
$$
\pi_\alpha :\Gamma_{-1}(\mathcal{H}_1)\overline{\otimes}
\Gamma_{-1}(\mathcal{H}_2)\longrightarrow M
\qquad\hbox{and}\qquad
J : VN(G)\longrightarrow M
$$
be the canonical
embeddings. The latter is 
a trace preserving one-to-one $\ast$-homomorphism.
We let $\E : \mathcal{M}\to B(L^2(\Sigma))$ denote the conditional 
expectation associated with $J$.

Let $\rho$ be the shift operator on $l^2_\Z$, as in Section 5. 
Then we may define completely positive, unital, trace preserving $\ast$-automorphisms :
\begin{align*}
\rho_1':=\Gamma_{-1}(Id_{H_{u_1}}\otimes \rho)\overline{\otimes}Id
\qquad\hbox{and}\qquad
\rho_2':=Id\overline{\otimes}\Gamma_{-1}(Id_{H_{u_2}}\otimes \rho).
\end{align*}
It is plain that they commute. 
In addition, for all $t\in G$ and $i=1,2$, 
we have $\alpha(t)\rho_i'=\rho_i'\alpha(t)$. 
Now thanks to \cite[Theorem 4.4.4 p. 149]{Sunder}, we 
may obtain $\tilde{\rho}_1$ and $\tilde{\rho}_2$ as in 
Theorem \ref{th application sunder}. Then
we have, for all $x\in \mathcal{H}_1$ and $y\in \mathcal{H}_2$,
\begin{align*}
\rho_1'(\omega_1(x)\otimes y)= \omega _1(Id\overset{2}{\otimes}\rho(x))\otimes y
\quad\hbox{and}\quad
\rho_1'(\omega_1(x)\otimes y)=x\otimes \omega _2(Id\overset{2}{\otimes}\rho(y)).
\end{align*}
We now introduce the two self-adjoint
symmetries 
$$
d_1=\pi_\alpha(\omega_1(\dot{\delta}_e\otimes \epsilon_0)\otimes 1)
\qquad\hbox{and}\qquad
d_2=\pi_\alpha(1\otimes\omega_2(\dot{\delta}_e\otimes \epsilon_0)).
$$
Then for all $t\in G$, we have
\begin{align*}
\alpha(t)(\omega_1(\dot{\delta}_e\otimes \epsilon_0)\otimes 1)=\omega_1(\dot{\delta}_t\otimes \epsilon_0)\otimes 1
\quad\hbox{and}\quad
\alpha(t)(1\otimes\omega_2(\dot{\delta}_e\otimes \epsilon_0))=1\otimes\omega_2(\dot{\delta}_t\otimes \epsilon_0).
\end{align*}
Finally we define the  trace preserving $*$-automorphisms 
$U_1,U_2$ on $\Gamma_{-1}(\mathcal{H}_1)\overline{\otimes}\Gamma_{-1}(\mathcal{H}_2)\rtimes_\alpha  G$ by 
\begin{align*}
U_i:x\mapsto d_i^*\tilde{\rho}_i(x)d_i.
\end{align*}
We claim that $U_1$ and $U_2$ commute. 

To check this, we first remark that:
\begin{align*}
d_1d_2&=\pi_\alpha(\omega_1(\dot{\delta}_e\otimes \epsilon_0)\otimes 1)\pi_\alpha(1\otimes\omega_2(\dot{\delta}_e\otimes \epsilon_0))\\
&=\pi_\alpha(\omega_1(\dot{\delta}_e\otimes \epsilon_0)\otimes 1\times 1\otimes\omega_2(\dot{\delta}_e\otimes \epsilon_0))\\
&=\pi_\alpha( 1\otimes\omega_2(\dot{\delta}_e\otimes \epsilon_0)\times \omega_1(\dot{\delta}_e\otimes \epsilon_0)\otimes 1)\\
&=d_2d_1.
\end{align*}
Second, for all $x\in \Gamma_{-1}(\mathcal{H}_1)
\overline{\otimes}\Gamma_{-1}(\mathcal{H}_2)$, we have
\begin{align*}
\tilde{\rho}_1\tilde{\rho}_2(\pi_\alpha(x))=\pi_\alpha(\rho_1'\rho_2'(x))=\pi_\alpha(\rho_2'\rho_1'(x))=\tilde{\rho}_2\tilde{\rho}_1(\pi_\alpha(x)),
\end{align*}
and for all $t\in G$, we have
\begin{align*}
\tilde{\rho}_1\tilde{\rho}_2(J(\lambda(t)))=J(\lambda(t))=\tilde{\rho}_2\tilde{\rho}_1(J(\lambda(t))).
\end{align*}
We deduce that $\tilde{\rho}_1$ and $\tilde{\rho}_2$ commuting. 
For all $x\in \Gamma_{-1}(\mathcal{H}_1)\overline{\otimes}
\Gamma_{-1}(\mathcal{H}_2)\rtimes_\alpha  G$, we obtain:
\begin{align*}
U_1(U_2(x))&=U_1(d_2\tilde{\rho}_2(x)d_2)\\
&=d_1\tilde{\rho}_1(d_2\tilde{\rho}_2(x)d_2)d_1\\
&=d_1\tilde{\rho}_1(d_2)\tilde{\rho}_1(\tilde{\rho}_2(x))\tilde{\rho}_1(d_2)d_1\\
&=d_1\tilde{\rho}_1(\pi_\alpha(1\otimes\omega_2(\dot{\delta}_e\otimes \epsilon_0)))\tilde{\rho}_2(\tilde{\rho}_1(x))\tilde{\rho}_1(\pi_\alpha(1\otimes\omega_2(\dot{\delta}_e\otimes \epsilon_0)))d_1\\
&=d_1\pi_\alpha(\rho'_1(1\otimes\omega_2(\dot{\delta}_e\otimes \epsilon_0)))\tilde{\rho}_2(\tilde{\rho}_1(x))\pi_\alpha(\rho'(1\otimes\omega_2(\dot{\delta}_e\otimes \epsilon_0)))d_1\\
&=d_1\pi_\alpha(1\otimes\omega_2(\dot{\delta}_e\otimes \epsilon_0))\tilde{\rho}_2(\tilde{\rho}_1(x))\pi_\alpha(1\otimes\omega_2(\dot{\delta}_e\otimes \epsilon_0))d_1\\
&=d_1d_2\tilde{\rho}_2(\tilde{\rho}_1(x))d_2d_1\\
&=d_2d_1\tilde{\rho}_2(\tilde{\rho}_1(x))d_1d_2\\
&\vdots \text{ same computation}\\
&=U_2(U_1(x))
\end{align*}
With this commutation property in hands, 
the proof of 
Theorem (\textbf{D}) is obtained by arguments similar to the ones 
in Section 5.

As in the previous subsection (Schur multipliers)
we mention that 
applying Lemma \ref{lem superdilatation}, we obtain the
following consequence of Theorem (\textbf{D}).

\begin{theorem}
Let $G$ be a unimodular locally compact group and 
let $T_{u_1},\ldots, T_{u_n}$ be self-adjoint, unital, completely positive Fourier multipliers on $VN(G)$. Then for all $1\leq p<\infty$,
there exist a tracial von Neumann algebra $(M,\tau)$,
a commuting $n$-tuple $(U_1,\dots,U_n)$ of surjective isometries  on $L^p(M)$,
and two contractions $J: L^p(VN(G))\to L^p(M)$ and 
$Q:L^p(M)\to L^p(VN(G))$ such that for all $k_1,\ldots,k_n
\in{\mathbb N}_0$,
\begin{align*}
T_{u_1}^{k_1}\cdots T_{u_n}^{k_n}=QU_1^{k_1}\cdots U_n^{k_n}J
\qquad on \ L^p(VN(G)).
\end{align*}
\end{theorem}

\vskip 1cm

\begin{large}
Acknowledgments:\end{large} I would like to thank Christian Le Merdy, my thesis supervisor for all his support and his help.

The LmB receives support from
the EIPHI Graduate School (contract ANR-17-EURE-0002)

\vskip 1cm
\bibliographystyle{abbrv}
\bibliography{bib}

\end{document}